\documentclass[11pt]{article}

\usepackage{tikz-cd}
\usetikzlibrary{decorations.pathmorphing}
\usepackage{latexsym}
\usepackage{amssymb}
\usepackage{amsthm}
\usepackage{amscd}
\usepackage{amsmath}
\usepackage{mathrsfs}
\usepackage{graphicx}
\usepackage{hyperref}
\usepackage{shuffle}

\usepackage[all]{xy}
\input xy \xyoption{frame}

\usepackage[colorinlistoftodos]{todonotes}
\usetikzlibrary{chains,scopes,decorations.markings}

\theoremstyle{definition}
\newtheorem* {theorem*}{Theorem}
\newtheorem* {conjecture*}{Conjecture}
\newtheorem{theorem}{Theorem}[section]

\theoremstyle{definition}
\newtheorem{hypothesis}[theorem]{Hypothesis}
\theoremstyle{definition}

\newtheorem* {example*}{Example}

\newtheorem{lemma}[theorem]{Lemma}
\theoremstyle{definition}
\newtheorem{definition}[theorem]{Definition}
\theoremstyle{definition}

\newtheorem{conjecture}[theorem]{Conjecture}
\newtheorem{proposition}[theorem]{Proposition}
\newtheorem{corollary}[theorem]{Corollary}

\newtheorem {remark}[theorem]{Remark}
\theoremstyle{definition}
\newtheorem {example}[theorem]{Example}
\theoremstyle{definition}

\theoremstyle{definition}

\theoremstyle{definition}

\theoremstyle{definition}

\def\modu{\ (\mathrm{mod}\ }

\def\({\left(}
\def\){\right)}

\newcommand{\CC}{\mathbb{C}}
\newcommand{\QQ}{\mathbb{Q}}

\newcommand{\cR}{\mathcal{R}}
\newcommand{\V}{\mathcal{V}}
\newcommand{\cS}{\mathcal{S}}

\newcommand{\cD}{\mathcal{D}}

\def\cS{\mathcal{S}}

\def\NN{\mathbb{N}}

\def\CC{\mathbb{C}}

\def\ZZ{\mathbb{Z}}

\def\spanning{\textnormal{-span}}

\newcommand{\cM}{\mathcal{M}}
\newcommand{\cN}{\mathcal{N}}

\def\fk{\mathfrak}

\def\barr{\begin{array}}
\def\earr{\end{array}}
\def\ba{\begin{aligned}}
\def\ea{\end{aligned}}
\def\be{\begin{equation}}
\def\ee{\end{equation}}

\def\Cyc{\mathrm{Cyc}}

\def\qquand{\qquad\text{and}\qquad}
\def\quand{\quad\text{and}\quad}
\def\qquord{\qquad\text{or}\qquad}

\newcommand{\Sym}{\operatorname{Sym}}

\def\cH{\mathcal H}
\def\cM{\mathcal M}
\def\DesR{\mathrm{Des}_R}
\def\DesL{\mathrm{Des}_L}
\def\DesV{\mathrm{Des}_V}

\def\hs{\hspace{0.5mm}}

\def\id{\mathrm{id}}
\def\PP{\mathbb{P}}
\def\dash{\hs\text{---}\hs}

\def\fkS{\fk S}

\def\ben{\begin{enumerate}}
\def\een{\end{enumerate}}

\def\cE{\mathcal E}
\def\cF{\mathcal F}

\def\hs{\hspace{0.5mm}}

\def\ellhat{\hat\ell}

\def\x{\textbf{x}}

\newcommand{\xRightarrow}[2][]{\ext@arrow 0359\Rightarrowfill@{#1}{#2}}

\renewcommand{\O}{\operatorname{O}}

\newcommand{\rank}{\operatorname{rank}}

\newcommand{\cA}{\mathcal{A}}

\def\iR{\hat\cR}

\def\ellhat{\hat\ell}

\def\iS{\hat \fkS}

\def\tS{\tilde S}

\def\tI{\tilde I}
\def\cG{\mathcal{G}}

\def\Inv{\operatorname{Inv}}
\def\HA{\cA_{\textsf{Hecke}}}
\def\iF{\hat{F}}

\def\tF{F}
\def\QSym{\textsf{QSym}}
\def\Sym{\textsf{Sym}}
\def\zetaq{\zeta_{\QSym}}
\def\Par{\textsf{Par}}
\def\tI{\tilde I}
\def\iPhi{\Upsilon}

\def\Refl{\textsf{Ref}}
\def\Cov{\textsf{Cov}}
\def\dotequals{\mathbin{\overset{_\bullet} =}}
\def\V{\textsf{Vir}}

\newcommand{\arc}[2]{ \ar @/^#1pc/ @{-} [#2] }
\def\arcstop{\endxy\ }
\def\arcstart{\ \xy<0cm,-.06cm>\xymatrix@R=.1cm@C=.2cm }
\newcommand{\arcstartc}[1]{\ \xy<0cm,-.15cm>\xymatrix@R=.1cm@C=#1cm}

\def\diagramAB{
\arcstart
{
*{\circ}     & *{\circ}
}
\arcstop}

\def\diagrambAC{
\arcstart
{
*{\bullet}    \arc{.5}{r}  & *{\circ} & *{\circ}
}
\arcstop
}
\def\diagramBaC{
\arcstart
{
*{\circ}    \arc{.5}{r}  & *{\bullet} & *{\circ}
}
\arcstop
}
\def\diagramcBA{
\arcstart
{
*{\bullet}    \arc{.5}{rr}  & *{\circ} & *{\circ}
}
\arcstop
}

\def\diagramACb{
\arcstart
{
*{\circ}     & *{\circ}  \arc{.5}{r} & *{\bullet}
}
\arcstop
}
\def\diagramAcB{
\arcstart
{
*{\circ}     & *{\bullet}  \arc{.5}{r} & *{\circ}
}
\arcstop
}

\def\diagramCBa{
\arcstart
{
*{\circ}     \arc{.5}{rr}   & *{\circ} & *{\bullet}
}
\arcstop
}

\def\diagramBaDc{
\arcstart
{
*{\circ}  \arc{.5}{r}   & *{\bullet}    & *{\circ} \arc{.5}{r}& *{\bullet}
}
\arcstop
}
\def\diagramBadC{
\arcstart
{
*{\circ}  \arc{.5}{r}   & *{\bullet}    & *{\bullet} \arc{.5}{r}& *{\circ}
}
\arcstop
}
\def\diagrambAdC{
\arcstart
{
*{\bullet}  \arc{.5}{r}   & *{\circ}    & *{\bullet} \arc{.5}{r}& *{\circ}
}
\arcstop
}
\def\diagrambADc{
\arcstart
{
*{\bullet}  \arc{.5}{r}   & *{\circ}    & *{\circ} \arc{.5}{r}& *{\bullet}
}
\arcstop
}

\def\diagramCDab{
\arcstart
{
*{\circ}  \arc{.5}{rr}   & *{\circ} \arc{.5}{rr}   & *{\bullet} & *{\bullet}
}
\arcstop
}

\def\diagramCdaB{
\arcstart
{
*{\circ}  \arc{.5}{rr}   & *{\bullet} \arc{.5}{rr}   & *{\bullet} & *{\circ}
}
\arcstop
}
\def\diagramcDAb{
\arcstart
{
*{\bullet}  \arc{.5}{rr}   & *{\circ} \arc{.5}{rr}   & *{\circ} & *{\bullet}
}
\arcstop
}

\def\diagramcdAB{
\arcstart
{
*{\bullet}  \arc{.5}{rr}   & *{\bullet} \arc{.5}{rr}   & *{\circ} & *{\circ}
}
\arcstop
}

\def\diagramDcBa{
\arcstart
{
*{\circ}  \arc{.5}{rrr}   & *{\bullet} \arc{.25}{r}   & *{\circ} & *{\bullet}
}
\arcstop
}

\def\diagramdCbA{
\arcstart
{
*{\bullet}  \arc{.5}{rrr}   & *{\circ} \arc{.25}{r}   & *{\bullet} & *{\circ}
}
\arcstop
}

\def\diagramDCba{
\arcstart
{
*{\circ}  \arc{.5}{rrr}   & *{\circ} \arc{.25}{r}   & *{\bullet} & *{\bullet}
}
\arcstop
}

\def\diagramdcBA{
\arcstart
{
*{\bullet}  \arc{.5}{rrr}   & *{\bullet} \arc{.25}{r}   & *{\circ} & *{\circ}
}
\arcstop
}

\usepackage{fullpage}

\numberwithin{equation}{section}
\allowdisplaybreaks[1]

\makeatletter
\renewcommand{\@makefnmark}{\mbox{\textsuperscript{}}}
\makeatother

\begin{document}
\title{Affine transitions for involution Stanley symmetric functions}
\author{
Eric MARBERG\thanks{
Email: \tt emarberg@ust.hk
}
\and
Yifeng ZHANG\thanks{
Email: \tt yzhangci@connect.ust.hk
}
}
\date{Department of Mathematics \\ Hong Kong University of Science and Technology}

\maketitle

\begin{abstract}
We study a family of symmetric functions $\hat F_z$
indexed by involutions $z$ in the affine symmetric group.
These power series
are analogues of
Lam's affine Stanley symmetric functions
and generalizations of the involution Stanley symmetric functions introduced by 
Hamaker, Pawlowski, and the first author.
Our main result is to prove a transition formula
for $\hat F_z$
which can be used to define an affine involution 
analogue of the
Lascoux-Sch\"utzenberger tree.
Our proof of this formula relies on Lam and Shimozono's
transition formula for affine Stanley symmetric functions
and some new technical properties of 
the strong Bruhat order on affine permutations.
\end{abstract}

\setcounter{tocdepth}{2}

\section{Introduction}

The notion of the \emph{Stanley symmetric function} $F_\pi$ of a permutation $\pi$
dates to work of Stanley \cite{Stan} in the 1980s.
These symmetric generating functions are of interest as
the stable limits of the \emph{Schubert polynomials} $\fkS_\pi$,
which represent the cohomology classes of certain orbit closures in the complete flag variety.
They are also a useful tool for calculating the number of reduced expressions for permutations.

Several variations of Stanley's construction have since appeared in the literature.
For example, around 2006, Lam \cite{Lam} introduced a larger family of symmetric functions $F_\pi$ indexed by permutations
$\pi$ in the affine symmetric group $\tS_n$.
These \emph{affine Stanley symmetric functions} represent
 cohomology classes for the affine Grassmannian  \cite[\S7]{Lam2008}.
More recent work of Hamaker, Pawlowski, and the first author \cite{HMP1,HMP4,HMP3}
studies the so-called \emph{involution Stanley symmetric functions}, which are indexed by self-inverse permutations $z=z^{-1}$ in the finite symmetric group $S_n$.
These power series arise as  the stable limits of certain
\emph{involution Schubert polynomials}
that coincide (up to a scalar factor) with polynomials
introduced by Wyser and Yong \cite{WY}
to represent the cohomology classes of orbit closures of the orthogonal group
acting on the complete flag variety.

The subject of this article is a family of symmetric functions $\iF_z$ indexed by affine involutions $z =z^{-1}\in \tS_n$,
generalizing both of the preceding constructions. 
(For a diagrammatic summary of the relationships between our new family 
and other kinds of Stanley symmetric functions, see Figure~\ref{fig1}.)
Our first results concern several equivalent definitions and basic properties of these ``affine'' involution Stanley symmetric functions.
For example, we prove that each $\iF_z$ has a triangular expansion
into both monomial symmetric functions and affine Schur functions, and we identify 
the leading terms in these decompositions.
Most of the results in this paper focus on algebraic and combinatorial properties of $\iF_z$, but
we expect that these power series are related to the geometry of affine analogues of certain symmetric varieties.

\begin{figure}[h]  
{\small
\[
\begin{tikzcd}
\boxed{\barr{c} \iF_z\ (z=z^{-1} \in \tS_n) \\ \text{affine involution Stanley}\earr}    
\arrow[d, Rightarrow] & 
\boxed{\barr{c} \iF_z\ (z=z^{-1} \in S_n) \\ \text{involution Stanley}\earr} \arrow[l, hook]   \arrow[ddr, Rightarrow]
 \arrow[d, Rightarrow]  & 
\boxed{\barr{c} \iS_z\ (z=z^{-1} \in S_n) \\ \text{involution Schubert}\earr} \arrow[l, squiggly] 
\arrow[d, Rightarrow]
\\
\boxed{\barr{c} F_\pi\ (\pi \in \tS_n) \\ \text{affine Stanley}\earr}
\arrow[d, Rightarrow]
 & 
\boxed{\barr{c} F_\pi\ (\pi \in S_n) \\ \text{Stanley}\earr}\arrow[d, Rightarrow]
\arrow[l, hook]    & 
\boxed{\barr{c} \fkS_\pi\ (\pi \in S_n) \\ \text{Schubert}\earr} \arrow[l, squiggly, crossing over] 
\\
\boxed{\barr{c} F_\lambda\ (\lambda_1 \leq n-1 ) \\ \text{affine Schur}\earr}
 & 
\boxed{\barr{c} s_\lambda\ (\lambda_i \leq n-i) \\ \text{Schur}\earr}
\arrow[l, hook] 
  & 
\boxed{\barr{c} P_\lambda\ (\lambda_i \leq n-2i+1) \\ \text{Schur $P$-functions}\earr}
\arrow[l, Rightarrow]
\end{tikzcd}
\]
}
\caption{Families of polynomials and symmetric functions of interest, with the new power series considered here in the top left.
The arrow
$\hookrightarrow$ means ``is a special case of'' while $\leadsto$ means ``has stable limit'' and $\Rightarrow$ means ``expands positively into.''
}
\label{fig1}
\end{figure}

The Stanley symmetric functions studied in \cite{HMP4,Lam,Stan} are noteworthy for their positivity properties.
For finite permutations $\pi \in S_n$,
the power series $\tF_\pi$ is always \emph{Schur positive}, i.e., an $\NN$-linear combination of Schur functions $s_\lambda$ \cite{EG}. Similarly, involution Stanley symmetric function indexed by finite permutations are $\NN$-linear combination of
\emph{Schur $P$-functions} \cite{HMP4}.
The Stanley symmetric functions $\tF_\pi$ indexed by affine permutation $\pi \in \tS_n$, while not always Schur positive, are at least ``affine Schur positive'' \cite{Lam2008} (see Section~\ref{aff-sect}).

One way to prove these positivity properties is via \emph{transition equations}: certain families of identities relating 
sums of Stanley symmetric functions indexed by Bruhat covers of a given permutation. 
Lam and Shimozono described transition equations for affine Stanley symmetric functions in \cite{LamShim}.
Transition equations for involution Stanley symmetric functions are derived in \cite{HMP3,HMP4,HMP5}.
Our results in Section~\ref{trans-sect} show how to extend the latter formulas to the affine case.
We expect that these identities will be useful in studying
the positivity properties of affine involution Stanley symmetric functions,
which are not yet well understood. 

In the usual Bruhat order on finite or affine permutations,
each covering relation arises from right multiplication by some transposition,
so every covering relation is naturally labeled by a reflection.
The statement of our affine transition formula depends on certain
\emph{covering transformations} $\tau^n_{ij}$ that play a role analogous to
right multiplication by a reflection in the Bruhat order restricted to involutions in $\tS_n$.
This paper is a sequel to \cite{M}, which develops the basic properties of the affine covering transformations $\tau^n_{ij}$.
Once these operators are identified,
the statement of
our transition formula is the obvious affine analogue of 
 formulas in \cite{HMP3,HMP4} for finite involutions.
However, our proofs requires new methods,
as inductive arguments in \cite{HMP3,HMP4}
that rely on the finiteness of $S_n$ 
cannot be applied to the affine case.
Specifically, we prove our transition formula as a corollary
of two technical theorems about the Bruhat order on affine involutions,
which we refer to as the \emph{covering property} (Theorem~\ref{tau-thm1}) and the \emph{toggling property} (Theorem~\ref{tau-thm3}).
Our proofs of these results depend on some computer calculations.

The involution Stanley symmetric functions studied in \cite{HMP4,HMP5}
come in two families: one indexed by all involutions in the finite symmetric group $S_n$
and related to the geometry of the orbits of the orthogonal group acting on the
complete flag variety, the other indexed by fixed-point-free involutions
and related to the geometry of the orbits of the symplectic group in the same space.
The power series studied in this paper are the affine analogues of the first family of 
symmetric functions. However, the second family also has an affine generalization and
 a similar transition formula. The properties
of such ``fixed-point-free involution Stanley symmetric functions'' 
are explored by the second author in a sequel to this paper \cite{Zhang}.

To conclude this introduction, we sketch an outline of the rest of this article.
The next section reviews some preliminaries on affine permutations and Stanley symmetric functions.
Section~\ref{invol-sect}
introduces our new family of \emph{affine involution Stanley symmetric functions} and surveys their basic properties.
In Section~\ref{trans-sect}, we discuss the affine transition formula of Lam and Shimozono 
and prove its involution analogue using the aforementioned covering and toggling properties. 
Section~\ref{tau-sect1}, finally, contains the proofs of 
our technical results
about
the Bruhat order on affine involutions.

\subsection*{Acknowledgements}

This work was partially supported by RGC Grant ECS 26305218.
We thank
Brendan Pawlowski for helpful conversations during the development of this paper.

\section{Affine permutations}\label{aff-sect}

Let $n$ be a positive integer. Write $\ZZ$ for the set of integers and define $[n] = \{1,2,\dots,n\}$.
Let $\NN = \{0,1,2,\dots\}$ and $\PP=\{1,2,3,\dots\}$ be the sets of nonnegative and positive integers.

\begin{definition}
The \emph{affine symmetric group} $\tS_n$ is the group of bijections $\pi:\ZZ \to \ZZ$
satisfying $\pi(i+n) = \pi(i)+n$ for all $i \in \ZZ$ and $\pi(1) + \pi(2) + \dots +\pi(n) = 1 + 2 + \dots +n$.
\end{definition}

We refer to elements of $\tS_n$ as \emph{affine permutations}.
A \emph{window} for an affine permutation $\pi \in \tS_n$ is a sequence of the form $[\pi(i+1), \pi(i+2),\dots,\pi(i+n)]$ where $i \in \ZZ$.
An element $\pi \in \tS_n$ is uniquely determined by any of its windows,
and a sequence of $n$ distinct integers is a window for some $\pi \in\tS_n$ if and only if the integers represent
each congruence class modulo $n$ exactly once.

Let $s_i $ for $i \in \ZZ$  be the unique element of $\tS_n$ that interchanges $i$ and $i+1$ while fixing every integer $j \notin \{i,i+1\} + n\ZZ$.
One has $s_i = s_{i+n}$ for all $i \in \ZZ$, and 
 $\{s_1,s_2,\dots,s_n\}$ generates the group $\tS_n$.
With respect to this generating set, $\tS_n$ is the affine Coxeter group of type $\tilde A_{n-1}$.
The parabolic subgroup $S_n = \langle s_1,s_2,\dots,s_{n-1}\rangle \subset \tS_n$
is the finite Coxeter group of type $A_{n-1}$; its elements
are
the permutations $\pi \in \tS_n$ with $\pi([n]) = [n]$.

A \emph{reduced expression} for $\pi \in \tS_n$ is a minimal-length factorization $\pi = s_{i_1}s_{i_2}\cdots s_{i_l}$.
The \emph{length} of $\pi \in \tS_n$, denoted $\ell(\pi)$, is the number of factors  in any of its reduced expressions.
The value of $\ell(\pi)$ is also the number of equivalence classes in the set
$
\Inv(\pi) = \{ (i,j) \in \ZZ \times \ZZ : i< j\text{ and }\pi(i)> \pi(j)\}
$
under the relation $\sim$
on $\ZZ\times \ZZ$ with $(a,b) \sim (a',b')$ if and only if $a-a' = b-b' \in n \ZZ$.

\begin{definition}
A reduced expression $\pi = s_{i_1}s_{i_2}\cdots s_{i_l}$ for an affine permutation
is \emph{cyclically decreasing} if $s_{i_j + 1} \neq s_{i_k}$ for all $1 \leq j < k \leq l$. 
An element $\pi \in \tS_n$ is cyclically decreasing if it has a cyclically decreasing reduced expression.
\end{definition}

Each cyclically decreasing element $\pi$ of the finite subgroup $S_n  \subset \tS_n$
has a unique reduced expression $s_{i_1}s_{i_2}\cdots s_{i_l}$ 
which is decreasing in the sense that $n > i_1 > i_2 > \dots >i_l >0$.
In general, an
affine permutation $\pi \in \tS_n$ may have more than one cyclically decreasing reduced expression.

\begin{definition}[Lam \cite{Lam}] \label{lam-def}
The (affine) Stanley symmetric function of $\pi \in \tS_n$ is 
\[ \tF_\pi = \sum_{\pi = \pi^1 \pi^2 \cdots } x_1^{\ell(\pi^1)} x_2^{\ell(\pi^2)} \cdots \in \ZZ[[x_1,x_2,\dots]]\]
where the sum is over all factorizations $\pi = \pi^1 \pi^2 \cdots$ 
of $\pi$
into countably many (possibly empty) cyclically decreasing factors
$\pi^i \in \tS_n$ such that $\ell(\pi) = \ell(\pi^1) + \ell(\pi^2) + \dots$.
\end{definition}

These functions are denoted $\tilde F_\pi$ in \cite{Lam}.
For $\pi \in S_n \subset \tS_n$, the power series $\tF_\pi$ 
 coincide (after an inversion of indices) with the symmetric functions introduced by Stanley in \cite{Stan}.
 
\begin{example}\label{stan-ex}
Suppose $n=4$ so that $s_1=s_5$.
There are two reduced expressions for the affine permutation $[0,3,6,1] = s_1s_2s_4s_3= s_1s_4s_2s_3 \in \tS_4$.
The distinct length-additive factorizations of this element into nontrivial cyclically decreasing factors are $(s_1)(s_2)(s_4)(s_3)$ and $(s_1)(s_4)(s_2)(s_3)$
and $(s_1s_4)(s_2)(s_3)$ and $(s_1)(s_2s_4)(s_3) = (s_1)(s_4s_2)(s_3)$ and $(s_1)(s_2)(s_4s_3)$,
so $\tF_{[0,3,6,1]} = 2m_{1^4} + m_{21^2}$,
where $m_\lambda$ denotes
the usual monomial symmetric function of a partition $\lambda$.

 \end{example}

Let $\pi \mapsto \pi^\circlearrowright$ denote the automorphism of $\tS_n$ with $s_i \mapsto s_i^\circlearrowright := s_{i+1}$.

\begin{proposition}[{Lam \cite[Proposition 18]{Lam}}]
If $\pi \in \tS_n$ then $\tF_{\pi^\circlearrowright} = \tF_\pi$.
\end{proposition}

One can  motivate the definition of $\tF_\pi$ using the theory of \emph{combinatorial coalgebras} from \cite{ABS}.
Define a combinatorial coalgebra $(C,\zeta)$ to be a graded, connected $\QQ$-coalgebra $C$ with a
linear map $\zeta : C \to \QQ$.
A morphism of combinatorial coalgebras $(C,\zeta) \to (C',\zeta')$ is a morphism of 
graded coalgebras $\phi: C \to C'$ satisfying $\zeta = \zeta'\circ \phi$.

For $\pi \in \tS_n$, we write $\pi \dotequals  \pi'\pi''$ to mean that $\pi',\pi'' \in \tS_n$, $\pi = \pi'\pi''$,
and $\ell(\pi) = \ell(\pi') + \ell(\pi'')$.

\begin{proposition}\label{coalg-prop}
The graded vector space $\QQ \tS_n$, in which $\pi \in \tS_n$ is homogeneous of degree $\ell(\pi)$,
is a graded, connected coalgebra with coproduct 
$\Delta(\pi) = \sum_{\pi \dotequals  \pi'\pi''} \pi' \otimes \pi''$ for $\pi \in \tS_n$.
\end{proposition}

\begin{proof}
The required axioms are easy to check directly using the associativity of group multiplication.
This coalgebra is just the graded dual of the $0$-Hecke algebra of $\tS_n$.
\end{proof}

Let $\QSym \subset \QQ[[x_1,x_2,\dots]]$ denote the (commutative) graded, connected Hopf algebra of quasi-symmetric functions over $\QQ$ (see \cite[\S3]{ABS}),
and write 
$\zetaq$
for the algebra homomorphism $\QSym \to \QQ$ which sets $x_1=1$ and $x_2=x_3=\dots= 0$.
Define $\zeta_{\textsf{CD}} : \QQ \tS_n \to \QQ$ to be the linear map with 
$\zeta_{\textsf{CD}}(\pi)=1$ if $\pi \in \tS_n$ is cyclically decreasing and with $\zeta_{\textsf{CD}}(\pi)=0$
for all other permutations $\pi \in \tS_n$.
The definition of $\tF_\pi$ is algebraically natural in view of the following.

\begin{proposition}\label{abs-prop}
The
linear map with $\pi \mapsto \tF_\pi$ for $\pi \in \tS_n$
 is the unique morphism of combinatorial coalgebras $(\QQ \tS_n,\zeta_{\textsf{CD}}) \to (\QSym,\zetaq)$.
 \end{proposition}

\begin{proof}
There exists a unique morphism $(\QQ \tS_n,\zeta_{\textsf{CD}}) \to (\QSym,\zetaq)$
by \cite[Theorem 4.1]{ABS}. The fact that $\pi \mapsto \tF_\pi$
gives this morphism follows by comparing \cite[Eq.\ (4.2)]{ABS} with 
Definition~\ref{lam-def}.
\end{proof}

\begin{corollary}[{Lam \cite[Theorem 12]{Lam}}]\label{coprod-cor}
If $\pi \in \tS_n$ then 
$ \Delta( \tF_\pi) = \sum_{\pi \dotequals  \pi'\pi''} \tF_{\pi'} \otimes \tF_{\pi''}$.
\end{corollary}

\begin{proof}
Apply the coalgebra morphism $\pi \mapsto \tF_\pi$
to both sides of $\Delta(\pi) = \sum_{\pi \dotequals  \pi'\pi''} \pi' \otimes \pi''$.
\end{proof}

Let $\Sym\subset \QSym$ denote the (commutative and cocommutative) Hopf subalgebra of symmetric functions over $\QQ$.
Let $\Par^n$ be the set of partitions with all parts less than $n$,
and define $\Sym^{(n)} = \QQ\spanning \{ m_\lambda : \lambda \in \Par^n\}.$
All of the quasi-symmetric functions $\tF_\pi$ turn out to be symmetric:

\begin{theorem}[{Lam \cite[Theorem 6]{Lam}}]
If $\pi \in \tS_n$ then $\tF_\pi \in \Sym^{(n)} \subset \Sym$.
\end{theorem}

To describe other features of $\tF_\pi$,
we recall some auxiliary data attached to permutations.

\begin{definition}\label{code-def}
The \emph{code} of an affine permutation $\pi \in \tS_n$ is the sequence $c(\pi) = (c_1,c_2,\dots,c_n)$ where $c_i $ is the number of integers $j\in \ZZ$ with $ i <j$ and $\pi(i)>\pi(j)$.
\end{definition}

Let $\pi \in \tS_n$ and write $c(\pi) = (c_1,c_2,\dots,c_n)$.
If $\pi(i)$ is minimal among $\pi(1),\pi(2),\dots,\pi(n)$, then we must have $c_i = 0$.
An integer $i \in \ZZ$ is a \emph{descent} of $\pi$ if $\pi(i) > \pi(i+1)$, i.e., if $\ell(\pi s_i) = \ell(\pi) -1$.
This holds
if and only if $c_i > c_{i+1}$, taking $c_{n+1} = c_1$.
If $i \in [n]$ is a descent of $\pi$ then 
\be\label{ccc-eq}
c(\pi s_i) = (c_1,\dots,c_{i-1}, c_{i+1}, c_i - 1,c_{i+2},\dots, c_n),
\ee interpreting indices cyclically as necessary.
By induction $|c(\pi)| := c_1 + c_2 + \dots+ c_n = \ell(\pi)$, and the map $\pi \mapsto c(\pi)$ is a bijection
$ \tS_n \to \NN^n - \PP^n$.

\begin{definition}\label{shape-def}
The \emph{shape} $\lambda(\pi)$ of $\pi \in \tS_n$ is the transpose of the partition that sorts $c(\pi^{-1})$.
\end{definition}

The map
$ \lambda : \tS_n \to \Par^n$
is surjective.
Write $<$ for the \emph{dominance order} on partitions.

\begin{theorem}[{Lam \cite[Theorem 13]{Lam}}]\label{uni-thm}
If $\pi \in \tS_n$ then $\tF_\pi \in m_{\lambda(\pi)} + \sum_{\mu < \lambda(\pi)} \NN m_\mu$.
\end{theorem}

This result implies that 
$ \QQ\spanning\{ \tF_\pi : \pi \in \tS_n \} = \QQ\spanning\{ \tF_\lambda : \lambda \in \Par^n\} = \Sym^{(n)}$.

\begin{example}
Suppose $n=4$ and $\pi = [-3,4,3,6]  \in \tS_4$ so that $\pi^{-1}=[5,0,3,2]$.
Then $c(\pi) = (0,2,1,2)$ and $c(\pi^{-1}) = (4,0,1,0)$
so $\lambda(\pi) = (2, 1,1,1)$ and $\lambda(\pi^{-1}) = (3,2)$,
and we have $\tF_\pi = m_{21^3} + 4m_{1^5}$
and $\tF_{\pi^{-1}} = m_{32} + 2m_{31^2} + 2m_{2^21} + 3m_{21^3} + 4m_{1^5}$.
\end{example}

Let $\DesR(\pi) = \{ s_i : i\in\ZZ\text{ is a descent of }\pi\} = \{ s \in \{ s_1,s_2,\dots,s_n\} : \ell(\pi s) < \ell(\pi)\}$
and $\DesL(\pi) = \DesR(\pi^{-1})$.
An element $\pi \in \tS_n$ is \emph{Grassmannian} if $\pi^{-1}(1)  < \pi^{-1}(2) < \dots < \pi^{-1}(n)$.
This occurs if and only if $\DesL(\pi) \subset \{ s_n\}$, or equivalently if $c(\pi^{-1})$ is weakly increasing.

\begin{definition}
The \emph{affine Schur function} $\tF_\lambda$ indexed by $\lambda \in \Par^n$
is the Stanley symmetric function
$\tF_\lambda := \tF_{\pi}$ where $\pi \in \tS_n$ is the unique Grassmannian element of shape $ \lambda$.
\end{definition}

Lam has shown that the symmetric functions $\tF_\pi$ are \emph{affine Schur positive} in the following sense:

\begin{theorem}[{Lam \cite[Corollary 8.5]{Lam2008}}] 
\label{+-thm}
$\NN\spanning\{ \tF_\pi : \pi \in \tS_n \} = \NN\spanning\{ \tF_\lambda : \lambda \in \Par^n\}.$
\end{theorem}

Affine Schur functions are not always \emph{Schur positive}; i.e., they do not 
necessarily expand as nonnegative linear combinations of ordinary Schur functions $s_\lambda$.
The Stanley symmetric functions indexed by finite permutations $\pi \in S_n \subset \tS_n$ do have this stronger positivity property, however:

\begin{theorem}[See \cite{EG,Lascoux}]
$\NN\spanning\{ \tF_\pi : \pi \in S_n \} \subset \NN\spanning\{ s_\lambda : \lambda \in \Par^n,\ \lambda \subset (n-1,\dots,2,1)\}.$
\end{theorem}

One can refine Theorems~\ref{uni-thm} and \ref{+-thm}.
Write $w \mapsto w^*$ for the unique group automorphism of $\tS_n$ with $s_i \mapsto s_i^* := s_{n-i}$ for $i \in \ZZ$, so that $s_n^* = s_n$.
If $\lambda \in \Par^n$ then there exists a unique Grassmannian permutation $\pi$ with $\lambda = \lambda(\pi)$,
and one defines $\lambda^* = \lambda(\pi^*)$.
In turn, let $\lambda'(\pi)= \lambda(\pi^{-1})^*$ for $\pi \in \tS_n$.
Finally define $<^*$ to be the partial order on $\Par^n$ with $\lambda <^* \mu$ if and only if $\mu^* < \lambda^*$.

\begin{theorem}[{Lam \cite[Theorem 21]{Lam}}]\label{schur-thm}
If $\pi \in \tS_n$ then 
\[ \tF_\pi \in \( F_{\lambda'(\pi)} + \sum_{ \lambda'(\pi) <^* \mu} \NN \tF_\mu\) \cap \( \tF_{\lambda(\pi)} + \sum_{\mu < \lambda(\pi)} \NN \tF_\mu\).\]
\end{theorem}

The affine Schur functions form a basis for $\Sym^{(n)}$, so there exists a unique linear involution $\omega^+ : \Sym^{(n)} \to \Sym^{(n)}$ with
$\omega^+(\tF_\lambda) = \tF_{\lambda^*}$ for all $\lambda \in \Par^n$.
This map can be defined directly in terms of the usual elementary, homogeneous, and monomial symmetric functions; see \cite[\S9]{Lam}.

\begin{theorem}[{Lam \cite[Theorem 15 and Proposition 17]{Lam}}]
\label{oom-thm}
 If $\pi \in \tS_n$ then $\omega^+(\tF_\pi) = \tF_{\pi^*} = \tF_{\pi^{-1}}$.
\end{theorem}

\section{Affine involutions}\label{invol-sect}

For integers $i<j \not \equiv i \modu n)$, let $t_{ij} \in \tS_n$ be the affine permutation
interchanging $i$ and $j$ while fixing all integers $k \notin \{i,j\} + n \ZZ$,
so that $t_{i,i+1} = s_i$.
Such permutations are precisely 
the \emph{reflections} in $\tS_n$, i.e., the elements conjugate to $s_i$ for some $i \in \ZZ$.

Let $\tI_n = \{ z \in \tS_n : z =z^{-1} \}$ be the set of involutions in $\tS_n$.
Each $z \in \tI_n$ is a product of commuting reflections, so
uniquely corresponds to the following data: a disjoint (possibly empty) collection of pairs $\{i < j\} \in \binom{[n]}{2}$
and for each pair an integer $m \in \ZZ$, such that 
$z$ is the product of the commuting reflections $t_{i,j+mn}$.
A useful graphical method of representing this data is through the \emph{winding diagram} of an involution:
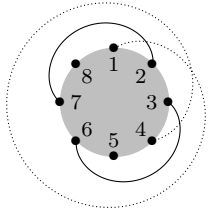
\[ 
\begin{tikzpicture}[baseline=0,scale=0.18,label/.style={postaction={ decorate,transform shape,decoration={ markings, mark=at position .5 with \node #1;}}}]
{
\draw[fill,lightgray] (0,0) circle (4.0);
\node at (2.44929359829e-16, 4.0) {$_\bullet$};
\node at (1.71450551881e-16, 2.8) {$_1$};
\node at (2.82842712475, 2.82842712475) {$_\bullet$};
\node at (1.97989898732, 1.97989898732) {$_2$};
\node at (4.0, 0.0) {$_\bullet$};
\node at (2.8, 0.0) {$_3$};
\node at (2.82842712475, -2.82842712475) {$_\bullet$};
\node at (1.97989898732, -1.97989898732) {$_4$};
\node at (2.44929359829e-16, -4.0) {$_\bullet$};
\node at (1.71450551881e-16, -2.8) {$_5$};
\node at (-2.82842712475, -2.82842712475) {$_\bullet$};
\node at (-1.97989898732, -1.97989898732) {$_6$};
\node at (-4.0, -4.89858719659e-16) {$_\bullet$};
\node at (-2.8, -3.42901103761e-16) {$_7$};
\node at (-2.82842712475, 2.82842712475) {$_\bullet$};
\node at (-1.97989898732, 1.97989898732) {$_8$};
\draw [-,>=latex,domain=0:100,samples=100,densely dotted] plot ({(4.0 + 4.0 * sin(180 * (0.5 + asin(-0.9 + 1.8 * (\x / 100)) / asin(0.9) / 2))) * cos(90 - (0.0 + \x * 4.95))}, {(4.0 + 4.0 * sin(180 * (0.5 + asin(-0.9 + 1.8 * (\x / 100)) / asin(0.9) / 2))) * sin(90 - (0.0 + \x * 4.95))});
\draw [-,>=latex,domain=0:100,samples=100] plot ({(4.0 + 2.0 * sin(180 * (0.5 + asin(-0.9 + 1.8 * (\x / 100)) / asin(0.9) / 2))) * cos(90 - (90.0 + \x * 1.35))}, {(4.0 + 2.0 * sin(180 * (0.5 + asin(-0.9 + 1.8 * (\x / 100)) / asin(0.9) / 2))) * sin(90 - (90.0 + \x * 1.35))});
\draw [-,>=latex,domain=0:100,samples=100] plot ({(4.0 + 2.0 * sin(180 * (0.5 + asin(-0.9 + 1.8 * (\x / 100)) / asin(0.9) / 2))) * cos(90 - (270.0 + \x * 1.35))}, {(4.0 + 2.0 * sin(180 * (0.5 + asin(-0.9 + 1.8 * (\x / 100)) / asin(0.9) / 2))) * sin(90 - (270.0 + \x * 1.35))});
}
\end{tikzpicture}
\]
Here, the numbers $1,2,\dots,n$ are arranged in order around a circle,
and a curve traveling $m$ times clockwise around the vertex 1 connects each of the chosen pairs $\{i,j\}$.
(We draw these curves in different styles for readability.)
The example represents the involution $z = t_{1,12} \cdot t_{3,6}\cdot t_{7,10}  \in \tI_8$.

There exists a
unique associative product $\circ: \tS_n \times \tS_n \to \tS_n$
with
$s_i \circ s_i =s_i$ for all $i \in \ZZ$ and 
with
$\pi' \circ \pi'' = \pi$ whenever $\pi \dotequals  \pi'\pi''$
\cite[\S7.1]{Humphreys}.
Fix $i \in \ZZ$ and $z \in \tI_n$. One can check that
\be s_i \circ z \circ s_i = \begin{cases} 
z &\text{if }z(i) > z(i+1) \\
zs_i =s_iz&\text{if }i = z(i) < z(i+1)=i+1 \\
s_izs_i &\text{otherwise}.
\end{cases}
\ee
It follows by induction that
$
\tI_n  = \{ \pi ^{-1} \circ \pi : \pi \in \tS_n\}
$,
so the set 
$\HA(z) := \{ \pi \in \tS_n : \pi^{-1} \circ \pi = z\}$
is nonempty.
Since $\ell(\pi) \leq \ell(\pi^{-1}\circ \pi)$ for all $\pi \in \tS_n$,
the set $\HA(z)$ is also necessarily finite.
Let $\cA(z)$ be the subset of minimal-length permutations in $\HA(z)$.
Following \cite{HMP2},
we refer to elements of $\cA(z)$ as \emph{atoms} for $z$ and to elements of $\HA(z)$
as \emph{Hecke atoms}.

\begin{definition}\label{aff-def}
The \emph{(affine) involution Stanley symmetric function} of $z \in \tI_n$ is
$ \iF_z = \sum_{\pi \in \cA(z)} \tF_\pi$.
\end{definition}

This is an affine generalization of the symmetric functions studied in \cite{HMP1,HMP2,HMP3,HMP4,HMP5},
which are defined by the same formula but with  $z$ restricted to the set 
$I_n := \tI_n \cap S_n$.
There are some noteworthy parallels between 
$\iF_z$ for $z \in I_n \subset \tI_n$ and 
$\tF_\pi$ for $\pi \in S_n \subset \tS_n$.
For example, the power series $\iF_z$ for $z \in I_n$ are the stable limits of 
certain \emph{involution Schubert polynomials} $\iS_z$ (see \cite{HMP1,WY}),
which represent the cohomology classes of the orbit closures of the orthogonal group
$\O_n(\CC)$ on the complete flag variety.
Whereas each $\tF_\pi$ for $\pi \in S_n$ is Schur positive, each $\iF_z$ for $z \in I_n$
is \emph{Schur $P$-positive}, i.e., a nonnegative integral linear combination of the \emph{Schur $P$-functions} $P_\mu \in \Sym$ (see \cite[\S A.3]{Stem}).
An overarching goal of this article is to understand the extent to which such parallels carry over to the affine case.

To understand the properties of $\iF_z$, we should describe the sets $\HA(z)$ and $\cA(z)$ more concretely.
Suppose $a_1,a_2,\dots,a_N$ are integers representing all congruence classes modulo $n$ at least once.
Define $[[ a_1,a_2,\dots,a_N]] \in \tS_n$ to be the affine permutation with a window 
given by reading the sequence $[a_1,a_2,\dots,a_N]$ left to right and omitting $a_j$ whenever $a_i \equiv a_j \modu n)$ for some $i<j$.
For example, if $n=5$ then $[[1,3,0,1,2,-1,4,8]] = [1,3,0,2,-1] = [3,0,2,-1,6] = [0,2,-1,6,8] \in \tS_5$.
 Let $z \in \tI_n$. 
Write $a_1<a_2 <\dots <a_l$ 
for the numbers $a \in  [n]$ with $a \leq z(a)$
and define
\be\label{amin-eq}
\alpha_{\min}(z) = [[ z(a_1), a_1,z(a_2), a_2,\dots,z(a_l),a_l]]^{-1} \in\tS_n.
\ee
Next write $b_1<b_2< \dots <b_l$
for the numbers $b \in [n]$ with $z(b) \leq b$ and define
\be\label{amax-eq}
\alpha_{\max}(z) = [[ b_1,z(b_1),b_2,z(b_2),\dots,b_l,z(b_l)]]^{-1} \in \tS_n.
\ee
The value of $l$ is the same  in both constructions.
Both $\alpha_{\min}(z)$ and $\alpha_{\max}(z)$ are in $\HA(z)$ by \cite[Proposition 6.8]{M}.

Below, we characterize the sets $\HA(z)$ and $\cA(z)$
as the equivalence classes of $\alpha_{\min}(z)$ or $\alpha_{\max}(z)$
under certain relations on $\tS_n$. These relations are generated by simple moves that rearrange three consecutive terms in a given window for the inverse of an affine permutation.

First let $\approx_\cA$ be the transitive closure of the symmetric relation on $\tS_n$
that has $u^{-1}\approx_\cA v^{-1} \approx_\cA w^{-1}$ whenever $u,v,w \in \tS_n$ have windows that are identical except 
in three consecutive positions where 
$u =[\cdots\ c,b,a\ \cdots],$ $v= [\cdots\ c,a,b\ \cdots], $ and $ w=[ \cdots\ b,c,a\ \cdots ]$
for some integers $a<b<c$. 
It is straightforward to show that this occurs if and only if 
there is a reduced expression $u^{-1} = s_{i} \cdot s_{i+1}\cdot s_i \cdot s_{j_1} \cdot s_{j_2} \cdots s_{j_l}$
such that $v^{-1} =  s_{i+1} \cdot s_i \cdot s_{j_1}\cdot s_{j_2} \cdots s_{j_l}$
and $w^{-1} = s_i \cdot s_{i+1}\cdot  s_{j_1}\cdot s_{j_2} \cdots s_{j_l}$ are also reduced expressions; compare with \cite[Lemma 6.3]{HMP2}.

\begin{theorem}[{\cite[Proposition 8.5]{M2021}}]
If $z \in \tI_n$ then \[\HA(z) =\{ \pi \in \tS_n :  \pi \approx_\cA \alpha_{\min}(z)\} =\{ \pi \in \tS_n :  \pi \approx_\cA \alpha_{\max}(z)\}.\]
\end{theorem}


Next let $\prec_\cA$ be the transitive closure of the relation on $\tS_n$ 
that has $v^{-1} \prec_\cA w^{-1}$ whenever $v$ and $w$ have windows that are identical except 
in three consecutive positions where 
 $v= [\cdots\ c,a,b\ \cdots] $ and $ w=[ \cdots\ b,c,a\ \cdots ]$
for some integers $a<b<c$. 

\begin{theorem}[{\cite[Theorem 6.14]{M}}]\label{atoms-thm}
Let $z \in \tI_n$. Restricted to $\cA(z)$, the relation $\prec_\cA$ is a graded partial order
and it holds that
$\cA(z) = \{ \pi \in \tS_n : \alpha_{\min}(z) \preceq_\cA \pi\} =  \{ \pi \in \tS_n : \pi \preceq_\cA \alpha_{\max}(z)\}.$
\end{theorem}

\begin{example}\label{aff-ex1}
Suppose $n=4$ and
\[z= 
\begin{tikzpicture}[baseline=0,scale=0.15,label/.style={postaction={ decorate,transform shape,decoration={ markings, mark=at position .5 with \node #1;}}}]
{
\draw[fill,lightgray] (0,0) circle (4.0);
\node at (2.4492935982947064e-16, 4.0) {$_\bullet$};
\node at (1.8369701987210297e-16, 3.0) {$_{1}$};
\node at (4.0, 0.0) {$_\bullet$};
\node at (3.0, 0.0) {$_{2}$};
\node at (2.4492935982947064e-16, -4.0) {$_\bullet$};
\node at (1.8369701987210297e-16, -3.0) {$_{3}$};
\node at (-4.0, -4.898587196589413e-16) {$_\bullet$};
\node at (-3.0, -3.6739403974420594e-16) {$_{4}$};
\draw [-,>=latex,domain=0:100,samples=100] plot ({(4.0 + 2.0 * sin(180 * (0.5 + asin(-0.9 + 1.8 * (\x / 100)) / asin(0.9) / 2))) * cos(90 - (180.0 + \x * 4.5))}, {(4.0 + 2.0 * sin(180 * (0.5 + asin(-0.9 + 1.8 * (\x / 100)) / asin(0.9) / 2))) * sin(90 - (180.0 + \x * 4.5))});
}
\end{tikzpicture}
= t_{3,8} = [1,2,8,-1] \in \tI_4.
\]
The elements of $\cA(z)$ are
\[\ba \alpha_{\min}(z) &= [1,2,8,3]^{-1} =[2,3,5,0] = [8,3,5,6]^{-1} \\&
\prec_\cA [5,8,3,6]^{-1} =[0,3,6,1]  \\&
\prec_\cA [5,6,8,3]^{-1} =[0,1,7,2]=[1,2,4,-1]^{-1} =\alpha_{\max}(z).
\ea\]
The elements of $\HA(z) -\cA(z) $ are 
$ [8,5,3,6]^{-1}$ and
$ [5,8,6,3]^{-1}$. In terms of reduced words, 
\[
\HA(z) = \{ s_2s_1s_0s_3,\ \ s_1s_0s_2s_3,\ \ s_0s_1s_2s_3,\ \ s_1s_0s_1s_2s_3,\ \
s_1s_2s_1s_0s_3\}. 
\]
Both $[2,3,5,0] = s_0s_1s_2s_3$ and $[0,1,7,2] = s_2s_1s_0s_3$ have a single reduced expression,
and it holds that
$\tF_{[2,3,5,0]} = m_{1^4}$ and $ \tF_{[0,1,7,2]} = m_{1^4} + m_{21^2} + m_{2^2} + m_{31}.$
We saw in Example~\ref{stan-ex}
that 
$\tF_{[0,3,6,1]} = 2m_{1^4} + m_{21^2}.$
Therefore $\iF_z = \iF_{[1,2,8,-1]} = 4m_{1^4} + 2m_{21^2} + m_{2^2} + m_{31}.$
\end{example}

For $y \in \tI_n$, let
$\Cyc(y) = \{ (a,b) \in \ZZ\times\ZZ : a \leq b = y(a)\}$.
Theorem~\ref{atoms-thm} suggests an efficient algorithm for generating the set of atoms for any involution in $\tS_n$.
One can also use the theorem to derive the following ``local'' characterization of $\cA(y)$,
which generalizes \cite[Theorem 2.5]{CJW}:

\begin{theorem}[{\cite[Theorem 7.6]{M}}]\label{local-thm}
Let $y \in \tI_n$ and $\pi \in \tS_n$.
Then $\pi \in \cA(y)$
if and only if for all $(a,b),(a',b') \in \Cyc(y)$, the following properties hold:
\begin{enumerate}
\item[(1)] If $a < b$ then $\pi(b) < \pi(a)$.
\item[(2)] If $a < a' \leq b' < b$ then we do not have $\pi(b) < \pi(a') < \pi(a)$ or $\pi(b) < \pi(b') < \pi(a)$.
\item[(3)] If $a < a'$ and $b < b'$ then $\pi(a) <\pi(b').$
\end{enumerate}
\end{theorem}

Suppose $E\subset \ZZ$ is a finite set of size $m$
representing distinct congruence classes modulo $n$. 
Write $\phi_E$ for the order-preserving bijection $[m] \to E$.
Theorem~\ref{local-thm} has this technical corollary:

\begin{corollary}\label{local-cor}
Suppose $y \in \tI_n$ and $\pi \in \cA(y)$ and $y(E) = E$.
Define $y' \in I_m$  and $\pi' \in S_m$
to be the unique permutations with $y'(i) < y'(j)$ if and only if $y\circ \phi_E(i) < y \circ \phi_E(j)$
and
$\pi'(i) < \pi'(j)$ if and only if $\pi\circ \phi_E(i) < \pi \circ \phi_E(j)$ for all $i,j \in [m]$.
Then  $\pi' \in \cA(y')$.
\end{corollary}

\begin{proof}
Write $\psi_E$ for the inverse of $\phi_E$.
Since $\Cyc(y') = \{ (\psi_E(a)+mn, \psi_E(b)+mn) : (a,b) \in \Cyc(y) \cap (E \times E),\ m \in \ZZ\}$,
the conditions in Theorem~\ref{local-thm} relative to $\pi'$ and $y'$ hold by construction.
\end{proof}

We note one other property of atoms and Hecke atoms.

\begin{lemma} Suppose $\pi \mapsto \pi^\dag$ is a length-preserving group automorphism of $\tS_n$.
For all $z \in \tI_n$, it then holds that
$\HA(z^\dag) = \{ \pi^\dag : \pi \in \HA(z)\}$ and $\cA(z^\dag) = \{\pi^\dag : \pi \in \cA(z)\}$.
\end{lemma}

\begin{proof}
It suffices to show that $(\pi')^\dag \circ (\pi'')^\dag = (\pi'\circ \pi'')^\dag$ for all $\pi',\pi'' \in\tS_n$.
This holds since  $(s_i)^\dag \circ (s_i)^\dag = (s_i)^\dag$
for $i \in [n]$
and $(\pi')^\dag \circ (\pi'')^\dag = (\pi'\pi'')^\dag$ if $\ell(\pi'\pi'')= \ell(\pi') + \ell(\pi'')$,
so the map $\tS_n \times \tS_n \xrightarrow{ \dag \times \dag} \tS_n \times \tS_n \xrightarrow{\ \circ\ } \tS_n \xrightarrow{ \dag^{-1}} \tS_n$
coincides with $\circ$ as it shares its defining properties.
\end{proof}

Both $\pi \mapsto \pi^\circlearrowright$ and $\pi \mapsto \pi^*$
are length-preserving, so Theorem~\ref{oom-thm} implies the following:

\begin{corollary}\label{aut-cor}
If $z \in \tI_n$ then $\iF_{z^\circlearrowright} = \iF_{z}$ and $\iF_{z^*} = \omega^+(\iF_z)$.
\end{corollary}

There is an analogue of Proposition~\ref{abs-prop} which motivates Definition~\ref{aff-def}.
 Define $\ell'(\pi)$ to be $n$ minus the number of orbits of $\pi \in\tS_n$ acting on $\ZZ/n\ZZ$.
The map $\ell' : \tS_n \to \NN$ is constant on conjugacy classes,
and 
if the congruence classes $i + n \ZZ$ and $i+1 + n \ZZ$ belong to distinct orbits under $\pi \in \tS_n$ then $\ell'(w s_i) = \ell'(w) + 1$.
 Let \be \ellhat(z) = \tfrac{1}{2}(\ell(z) +\ell'(z))\qquad\text{for $z \in \tI_n$.}\ee
By induction, the value of
$\ellhat(z)$ is always a nonnegative integer and $\ellhat(z) = \ell(\pi)$ for any $\pi \in \cA(z)$.
The homogeneous symmetric function $\iF_z$ therefore has degree $\ellhat(z)$.
Give $\QQ\tS_n$ the coalgebra structure from Proposition~\ref{coalg-prop} and write $\Delta$ for its coproduct.

\begin{proposition}\label{comod-prop}
The graded vector space $\QQ \tI_n$,
in which $z \in \tI_n$ is homogeneous of degree $\ellhat(z)$,
 is a graded right comodule for $\QQ \tS_n$ with coproduct 
$\hat\Delta: \QQ \tI_n \to \QQ\tI_n \otimes \QQ \tS_n$
given by the linear map
\[ \hat\Delta(z) = \sum_{\substack{(y,\pi) \in\tI_n \times \tS_n \\ \ellhat(z) = \ellhat(y) + \ell(\pi) \\ z = \pi^{-1} \circ y \circ \pi }} y \otimes \pi\qquad\text{for }z \in \tI_n.\]
\end{proposition}

\begin{proof}
It suffices to check that $(\hat\Delta \otimes \id) \circ \hat\Delta = (\id \otimes \Delta) \circ \hat\Delta$.
This holds by the associativity of $\circ$.
\end{proof}

Let $\fk F$ be the graded coalgebra morphism $\QQ \tS_n \to \Sym$ with $\pi \mapsto F_\pi$
for $\pi \in \tS_n$.
The graded vector space $\QQ \tI_n$ is then a graded right $\QSym$-comodule
with respect to the coproduct
$(\id \otimes \fk F) \circ \hat\Delta$.
The graded coalgebra $\QSym$ is automatically a graded right comodule for itself.

\begin{proposition}\label{comod-prop2}
The linear map with $z \mapsto \iF_z$ for $z \in \tI_n$ is the unique morphism  
of graded right $\QSym$-comodules $\QQ \tI_n \to \QSym$
satisfying $1 \mapsto 1 \in \QQ[[x_1,x_2,\dots]]$.
\end{proposition}

\begin{proof}
Let $\epsilon$ be the counit of $\QSym$.
Any such $\QSym$-comodule morphism  $\hat{\fk F} : \QQ \tI_n \to \QSym$
must satisfy
$1 \otimes \hat{\fk F}(z) = (\epsilon \otimes \id)\circ \Delta \circ \hat{\fk F}(z) =
(\epsilon \otimes \id) \circ (\hat{\fk F} \otimes \fk F) \circ  \hat\Delta(z) 
= 1 \otimes \iF_z$ for $z \in \tI_n$.
On the other hand, it follows from Corollary~\ref{coprod-cor} that the given map is 
a graded $\QSym$-comodule morphism.
\end{proof}

\begin{corollary}
If $z \in \tI_n$ then 
$ \Delta( \iF_z) = \sum \iF_y \otimes \tF_{\pi}$
where the sum is over all pairs $(y,\pi) \in \tI_n \times \tS_n$ with $\ellhat(z) = \ellhat(y) + \ellhat(\pi)$ and $z = \pi^{-1}\circ y \circ \pi$.
\end{corollary}

\begin{proof}
Apply the morphism in Proposition~\ref{comod-prop2} to both sides of the formula defining $\hat\Delta(z)$.
\end{proof}

The notions of codes, shapes, and so forth for affine permutations have analogues for involutions.
Most of the following definitions are affine generalizations of constructions from
 \cite{HMP1,HMP4}.

\begin{definition}
The \emph{involution code} of $z \in \tI_n$ is the sequence $\hat c(z) = (c_1,c_2,\dots,c_n)$
where $c_i$ is the number of integers $j \in \ZZ$ with $i<j$ and $z(i) > z(j)$ and $i \geq z(j)$.
\end{definition}

An integer $i \in \ZZ$ is a \emph{visible descent} of $z \in \tI_n$ if $z(i) > z(i+1)$ and $i \geq z(i+1)$.
Let \[
\DesV(z )= \{ s_i : i\in \ZZ\text{ is a visible descent of }z\}.\]

\begin{lemma}\label{desv-lem}
If $z \in \tI_n$ then $ \DesV(z) = \DesR(\alpha_{\min}(z))$ and $\hat c(z) = c(\alpha_{\min}(z)).$
\end{lemma}

For example, if $z =[1, -2, 7, 5, 4]= t_{3,7}t_{4,5} \in \tI_5$ and $\alpha_{\min}(z) = [1, 7, 3, 5, 4]^{-1}=[ 2, -2, 4, 6, 5]$,
then $\hat c(z) = c(\alpha_{\min}(z)) = (1, 0, 1, 2, 1)$ and $\DesV(z)  =\DesR(\alpha_{\min}(z))= \{1,4\}$.

\begin{proof}
Fix $z \in \tI_n$ and integers $i<j$. 
It is clear from the definition \eqref{amin-eq}
that $\alpha_{\min}(z)(i) > \alpha_{\min}(z)(j)$
if and only if either
$z(j) < z(i) < i < j$
or
$z(j) < i \leq z(i) < j$
or
$z(j) < i < j < z(i)$
or
$z(j) = i < z(i) = j$.
One of these cases occurs precisely when $z(i) > z(j)$ and $i \geq z(j)$.
We conclude that $\hat c(z) = c(\alpha_{\min}(z))$ and, taking $j=i+1$, that $ \DesV(z) = \DesR(\alpha_{\min}(z))$.
\end{proof}

It follows that every involution in $\tI_n -\{1\}$ has at least one visible descent.

\begin{corollary}\label{icode-cor}
Suppose $z \in \tI_n$ and $\hat c(z) = (c_1,c_2,\dots,c_n)$.
Then $\ellhat(z) =c_1+c_2+\dots+c_n $, and an integer
$i \in \ZZ$ is a visible descent of $z$ 
if and only if $c_i > c_{i+1}$, interpreting indices modulo $n$.
\end{corollary}

\begin{proof}
Since $\ellhat(z) = \ell(\alpha_{\min}(z))$, these properties are immediate from Lemma~\ref{desv-lem}.
\end{proof}

\begin{corollary}
For $z \in \tI_n$, the following are equivalent:
\ben
\item[(a)] $\DesV(z) \subset \{ s_n\}$.
\item[(b)] $\hat c(z)$ is weakly increasing.
\item[(c)] $\alpha_{\min}(z)^{-1}$ is Grassmannian.
\een
\end{corollary}

\begin{proof}
Given Lemma~\ref{desv-lem}, this is immediate from
the discussion after Definition~\ref{code-def}.
\end{proof}

This corollary suggests the property $\DesV(z) \subset \{ s_n\}$
as a natural definition for the ``involution'' analogue for a Grassmannian permutation.
However, the functions $\iF_z$ indexed by $z \in \tI_n$ with this property
fail to span $\QQ\spanning\{ \iF_z : z \in \tI_n\}$, although they are linearly independent.

Given $z \in \tI_n$ and $i \in \ZZ$,
define $z_{\downarrow(i)} \in \tI_n$ to be $zs_i$ if $zs_i=s_iz$ and $s_izs_i$ otherwise.
If $z(i)>z(i+1)$ then
$z_{\downarrow(i)}$ is the unique element of $\tI_n$ with $z_{\downarrow(i)} \neq s_i \circ z_{\downarrow(i)} \circ s_i = z$.

\begin{proposition}\label{icode-prop}
Let $z \in \tI_n$ and $\hat c(z) = (c_1,c_2,\dots,c_n)$.
Suppose $i \in [n]$ and $z(i)>z(i+1)$.
\ben
\item[(a)] If $i$ is a visible descent of $z$ then $\hat c(z_{\downarrow(i)}) = (c_1,\dots, c_{i-1}, c_{i+1},c_i - 1,c_{i+2},\dots, c_n)$.
\item[(b)] Assume $i$ is not a visible descent of $z$. Let $j \in [n]\setminus\{i\}$ be such that $j \equiv  z(i+1) \modu n)$.
\begin{itemize}
\item If $z(i + 1) = i+1$ then
$\hat c(z_{\downarrow(i)}) = (c_1,\dots, c_{i-1}, c_i -1 , c_{i+1}\dots, c_n)$.
\item If $z(i+1) \neq i + 1$ then
$\hat c(z_{\downarrow(i)}) = (c_1,\dots, c_{j-1}, c_j  - 1, c_{j+1},\dots, c_n)$.
\end{itemize}
\een
In both parts, indices are interpreted modulo $n$ as necessary. 
\end{proposition}

\begin{proof}
First suppose $i$ is a visible descent of $z$, so that $i$ is a descent of $\alpha_{\min}(z)$ by Lemma~\ref{desv-lem}.
From the formula \eqref{amin-eq},
it is straightforward to check that $\alpha_{\min}(z) s_i = \alpha_{\min}(z_{\downarrow(i)})$.
Part (a) therefore follows from \eqref{ccc-eq} and Lemma~\ref{desv-lem}.

To prove part (b), suppose $i$ is not a visible descent of $z$.
We then must have $z(i)>z(i+1)>i$. If $z(i+1)=i+1$, then the formula \eqref{amin-eq} implies that $\alpha_{\min}(z)^{-1}$ and $\alpha_{\min}(z_{\downarrow(i)})^{-1}$ have windows that are identical except in three consecutive positions where $\alpha_{\min}(z)^{-1}=[\dots,z(i),i,i+1,\dots]$ and $\alpha_{\min}(z_{\downarrow(i)})^{-1}=[\dots,i,z(i),i+1,\dots]$. Comparing these formulas, we see that $c_i(\alpha_{\min}(z_{\downarrow(i)}))=c_i(\alpha_{\min}(z))-1$ and $c_k(\alpha_{\min}(z_{\downarrow(i)}))=c_k(\alpha_{\min}(z))$ for $k\in [n]\setminus\{i\}$, so the desired result follows from Lemma~\ref{desv-lem}.

Alternatively, 
if $z(i+1)\neq i+1$, then $z(i+1)>i+1$ and the formula \eqref{amin-eq} implies that $\alpha_{\min}(z)^{-1}$ and $\alpha_{\min}(z_{\downarrow(i)})^{-1}$ have windows that are identical except in four consecutive positions where $\alpha_{\min}(z)^{-1}=[\dots,z(i),i,z(i+1),i+1,\dots]$ and $\alpha_{\min}(z_{\downarrow(i)})^{-1}=[\dots,z(i+1),i,z(i),i+1,\dots]$. Comparing these formulas, we see that $c_j(\alpha_{\min}(z_{\downarrow(i)}))=c_j(\alpha_{\min}(z))-1$ and $c_k(\alpha_{\min}(z_{\downarrow(i)}))=c_k(\alpha_{\min}(z))$ for $k \in [n]\setminus\{j\}$, so the desired result again follows from Lemma~\ref{desv-lem}.
\end{proof}

\begin{corollary} 
The involution code is an injective map $\hat c : \tI_n \to \NN^n - \PP^n$.
\end{corollary}

\begin{proof}
This follows by induction from Corollary~\ref{icode-cor} and Proposition~\ref{icode-prop}.
\end{proof}

\begin{definition}
The \emph{shape} $\mu(z)$ of $z \in \tI_n$ is the transpose of the partition  that sorts $\hat c(z)$.
\end{definition}

The maps $\hat c : \tI_n \to \NN^n - \PP^n$ and $\mu : \tI_n \to \Par^n$ are not surjective, and it is an open problem to characterize their images.
By results in  \cite[\S4]{HMP4}, the involution shape map $\mu$ restricts to a bijection from
$I_n = \tI_n \cap S_n$ to the set of strict partitions contained in $(n-1,n-3,n-5,\dots)$.
However, $\mu(z)$ is not necessarily strict for $z \in \tI_n - I_n$.

Recall the partition $\lambda(\pi)$ given in Definition~\ref{shape-def}.
As earlier, if $\lambda \in \Par^n$ then $\lambda^* = \lambda(\pi^*)$
where $\pi \in \tS_n$ is the unique Grassmannian permutation with $\lambda =\lambda(\pi)$, 
and we set $\lambda'(\pi) = \lambda(\pi^{-1})^*$.

\begin{lemma}\label{lambda-mu-lem}
If $z \in \tI_n$ then $ \mu(z) = \lambda(\alpha_{\max}(z))$ and $\mu(z)^* = \lambda'(\alpha_{\min}(z))$.
\end{lemma}

Since $*$ is involution, this implies that $ \lambda(\alpha_{\max}(z)) = \lambda(\alpha_{\min}(z)^{-1})$.

\begin{proof}
The second equality holds by Lemma~\ref{desv-lem},
so we just need to show that $ \mu(z) = \lambda(\alpha_{\max}(z))$.
Let $\pi = \alpha_{\max}(z)^{-1}$.
Let $\dots < b_{-1}<b_0< b_1 < b_2 < \dots $ be
the elements $\{b \in \ZZ : z(b) \leq b\}$ listed in order.
Assume $\{b_i : i \in \ZZ\} \cap [n] = \{ b_1,b_2,\dots,b_l\}$
where $l = \ell'(z)$,
and let $a_i = z(b_i)$.
By \eqref{amax-eq},
we have $\pi = [[b_1,a_1,b_2,a_2,\cdots, b_l,a_l]]$.
For each $e \in \ZZ$,
define 
\[p_z(e) = |\{ j \in\ZZ : j > e \geq z(j),\ z(e) > z(j)\}|
\quand
q_\pi(e) = |\{ j \in \ZZ : j > \pi^{-1}(e),\ e > \pi(j)\}|.
\]
Observe that $\mu(z)$ and $\lambda(\pi)$
are the transposes of the partitions sorting 
$p_z(1),p_z(2),\dots,p_z(n)$ and
$q_\pi(1),q_\pi(2),\dots,q_\pi(n)$, respectively.
It is straightforward to check that 
\[ \ba
p_z(b_i) = q_\pi(a_i) &= |\{ j \in \ZZ : i < j,\ a_i > a_j\}|
\ea
\]
for each $i \in [l]$ with $a_i \leq b_i$ and that 
\[\ba
p_z(a_i) &= |\{ j \in \ZZ : b_j > a_i  \geq a_j \}| =
1 + |\{ j \in \ZZ : b_j > a_i  > a_j \}|\\
 q_\pi(b_i)  &= |\{ j \in \ZZ : b_j > b_i \geq a_j \}| = 
1 + |\{ j \in \ZZ : b_j > b_i  > a_j \}|
\ea
\]
for each $i \in [l]$ with $a_i < b_i$.

Let $\cE$ be a set of pairwise disjoint 2-element subsets of $\ZZ$
with the property that $\{a,b\} \in \cE$ if and only if $\{a+n,b+n\} \in \cE$;
e.g., consider $\cE = \{ \{a,b\} : a < b = z(a)\}$.
In view of the previous paragraph,
to prove that $\mu(z) = \lambda(\pi)$,
it suffices to show that there exists a bijection
$\phi : \cE \to \cE$ 
with the following properties:
 \begin{itemize}
 \item[(i)] If $\phi(\{a,b\}) = \{a+mn,b+mn\}$ for $m \in \ZZ$
then $m=0$.
\item[(ii)] $\phi(\{a,b\}) = \{a',b'\}$ if and only if $\phi(\{a+n,b+n\}) = \{a'+n,b'+n\}$.
\item[(iii)] If $\phi(\{a,b\}) = \{a',b'\}$ where $a<b$ and $a'<b'$
then the sets
 \[R_\cE(a,b) := \{(x,y) \in \cE : x<b<y\}
\quand
L_\cE(a',b') := \{  (x,y) \in \cE : x < a'<y\}\]
have the same number of elements.
\end{itemize}
(Only the second two properties are needed in the argument that follows;
the first property is just a convenient normalization.)
 We refer to such a map $\phi$ 
 as a \emph{special matching} for $\cE$.

We argue by contradiction that a special matching exists for any choice of $\cE$.
(We suspect that this
might follow from known results, but we do not know of an 
appropriate reference, so include a self-contained proof.)
Let $\rank(\cE)$
be the number of equivalence classes in $\cE$ 
under the relation with $\{a,b\} \sim \{a',b'\}$
if and only if $a'-a = b'-b \in n\ZZ$.
 Suppose $\cE$ is of minimal rank such that no special matching exists.
No pair $\{a, b\} \in \cE$ with $a<b$ can be such that 
\[\{(x,y) \in \cE : x<a<y < b\}
=
\{  (x,y) \in \cE : a<x < b<y\}=\varnothing,\]
since then any special matching for
$\cE - \{ \{a+mn,b+mn\} : m \in \ZZ\}$ 
would uniquely extend to a special matching for $\cE$
with  $\{a,b\}$ as a fixed point.

It follows that there must exist $\{a,b\},\{a',b'\} \in \cE$
with $a<a'<b<b'$ 
such that no $\{x,y\} \in \cE$ has $a'<x<b$ or $a'<y<b$.
 First suppose $a'-a = b'-b \in n\ZZ$,
 and form $\cF$ from $\cE$ by removing 
$\{a+mn,b+mn\}$ for all $m \in \ZZ$.
A special matching for $\cF$ exists as $\rank(\cF) = \rank(\cE) - 1$,
and it is easy to see that this matching uniquely extends to a special matching on 
$\cE$ with $\{a,b\}$ as a fixed point, as in the previous paragraph.

Alternatively, assume $\{a,b\}+n\ZZ$ and $\{a',b'\}+n\ZZ$ are disjoint
and form $\cF$ from $\cE$ by removing $\{a+mn,b+mn\}$ and $\{a'+mn,b'+mn\}$ 
 and then adding $\{a+mn,b'+mn\}$ for all $m \in \ZZ$.
 A special matching $\psi$ for $\cF$ again exists by hypothesis.
 Define $\phi : \cE \to \cE$ to be the unique map
 with properties (i) and (ii) 
such that:
 \begin{itemize}
 
  \item $\phi(\{a,b\}) =  \{a'  ,b'  \}$.
  
  \item If $\psi(\{a,b'\}) = \{a,b'\}$ then $\phi(\{a',b'\}) = \{a,b\}$.
  
 \item If $\psi(\{x,y\}) = \{a,b'\}$ for $\{x,y\}\neq \{a,b'\}$ then $\phi(\{x,y\}) = \{a,b\}$.
 
 \item If $\psi(\{a,b'\}) = \{x,y\}$ for $\{x,y\}\neq \{a,b'\}$ then $\phi(\{a',b'\}) = \{x,y\}$.
 
 \item If $\{x,y\} \in \cF$ and neither $\{x,y\}$ nor $\psi(\{x,y\})$ belongs to 
$\{ \{a+mn,b'+mn\} : m \in \ZZ\},$ then $\phi(\{x,y\}) = \psi(\{x,y\})$.

\end{itemize}
We claim that $\phi$ is a special matching.
To show this, first observe that \[R_\cE(a,b) - \{\{a',b'\}\} = L_\cE(a',b') - \{ \{a,b\}\}\] by construction, so $|R_\cE(a,b)| = |L_\cE(a',b')|$ as needed.
Next, it is easy to see that $|R_\cE(a',b')| = |R_\cF(a,b')|$ and $|L_\cE(a,b)| = |L_\cF(a,b')|$.
Moreover, since we cannot have both $a<x<b$ and $a'<x<b'$
or both $a<y<b$ and $a'<y<b'$
for any $\{x,y\} \in \cE$,
it follows that $|R_\cE(x,y)| = |R_\cF(x,y)|$ and $|L_\cE(x,y)| = |L_\cF(x,y)|$ for all $\{x,y\} \in \cF - \{ \{a+mn,b'+mn\}$.
From these observations, it is straightforward to check that 
$|R_\cE(x,y)| = |L_\cE(x',y')|$ whenever $\phi(\{x,y\}) = \{x',y'\}$,
so $\phi$ is a special matching for $\cE$.
We deduce from this contradiction that 
 special matchings exist for any choice of $\cE$, which finally lets us conclude that
$\mu(z) = \lambda(\pi) = \lambda(\alpha_{\max}(z)^{-1})$.
\end{proof}

Recall the definition of the partial order $\prec_\cA$ from before Theorem~\ref{atoms-thm}.  

\begin{lemma}\label{i<-lem}
Let $\pi,\sigma \in \tS_n$. If $\pi \prec_\cA \sigma$ then $\lambda(\pi) < \lambda(\sigma)$  in dominance order and $\lambda'(\pi) <^* \lambda'(\sigma)$.
\end{lemma}

\begin{proof}
Suppose $\pi,\sigma \in \tS_n$ such that $\pi \prec_\cA \sigma$.
We may assume that $\pi^{-1}$ and $\sigma^{-1}$ have windows that are identical except 
in three consecutive positions $i-1$, $i$, and $i+1$, where 
 $\pi^{-1}= [\cdots, c,a,b, \cdots] $ and $ \sigma^{-1}=[ \cdots, b,c,a, \cdots ]$ for some integers $a<b<c$. 
 
 Then we have $c(\pi^{-1})_j=c(\sigma^{-1})_j$ for $j\notin\{ i,i-1,i+1\}+n\ZZ$ while $c(\pi^{-1})_{i-1}=c(\sigma^{-1})_{i}+1$, 
 $c(\pi^{-1})_{i}=c(\sigma^{-1})_{i+1}$,  
 $c(\pi^{-1})_{i+1}=c(\sigma^{-1})_{i-1}-1$, and 
 $c(\pi^{-1})_{i-1}>c(\pi^{-1})_{i+1}\ge c(\pi^{-1})_{i}$. It follows that the partition sorting $c(\pi^{-1})$ exceeds the partition sorting $c(\sigma^{-1})$ in dominance order.
 Since dominance order is reversed by taking transposes,
we have $\lambda(\pi)<\lambda(\sigma)$ as desired.
 
 Similarly, it holds under these conditions that  $c(\pi)_j=c(\sigma)_j$ for $j\notin\{ a,b\} + n\ZZ$ while $c(\sigma)_{a}-1=c(\pi)_a \geq  c(\pi)_{b}=c(\sigma)_{b} + 1$.
 Thus the partition sorting $c(\sigma)$ exceeds the partition sorting $c(\pi)$ in dominance order,
so we have
$\lambda(\sigma^{-1}) < \lambda(\pi^{-1})$
 which is equivalent to $\lambda'(\pi) <^* \lambda'(\sigma)$.
\end{proof}

\begin{theorem}\label{i<-thm}
Suppose $z \in \tI_n$. The following properties then hold:
\ben
\item[(a)] $\iF_z \in m_{\mu(z)} + \sum_{\nu < \mu(z)} \NN m_\nu \subset \Sym^{(n)}.$
\item[(b)] $ \iF_z \in \( F_{\mu(z)^*} + \sum_{ \mu(z)^* <^* \nu} \NN F_\nu\) \cap \( F_{\mu(z)} + \sum_{\nu < \mu(z)} \NN F_\nu\).$
\een
In both parts, the symbol $<$ denotes the dominance order on partitions.
\end{theorem}

\begin{proof}
Given Lemmas~\ref{lambda-mu-lem} and \ref{i<-lem},
the result is immediate from Theorems~\ref{uni-thm} and \ref{schur-thm}.\end{proof}

\begin{example}
Let $z = t_{3,8} = [1,2,8,-1] \in \tS_4$ as in Example~\ref{aff-ex1}
so that  
$\alpha_{\min}(z)=[2,3,5,0]$ and $\alpha_{\max}(z)=[1,2,4,-1]^{-1}$.
Then 
$
\hat c(z)  = c(\alpha_{\min}(z))= (1,1,2,0)
$
and
$
c(\alpha_{\max}(z)^{-1}) = (1,2,0,1)
$
so  
\[
\mu(z) = \lambda(\alpha_{\max}(z))=(3,1)
\qquand
\mu(z)^*= \lambda'(\alpha_{\min}(z)) = \lambda(\alpha_{\min}(z)^{-1})^* =(3,1)^*.
\]
The Grassmannian permutation $\pi \in \tS_4$ with $\lambda(\pi) = (3,1)$ is $\pi = [-3, 3, 4, 6]^{-1}$.
Since $\pi^* = [-1, 1, 2, 8 ]^{-1}$ has shape $\lambda(\pi^*) = (1,1,1,1)$, we have
$\mu(z)^* = (1,1,1,1)$.
This agrees with Theorem~\ref{i<-thm} as $\iF_{z} =  m_{1^4} + m_{21^2} + m_{2^2} + m_{31} = \tF_{1^4} + \tF_{21^2} + \tF_{31}.$
\end{example}

Some basic questions about involution Stanley symmetric functions remain open.
It is known that $\{ \iF_z : z \in I_n\subsetneq \tI_n\}$ spans 
exactly the same vector space as the set of Schur $P$-functions $P_\mu$
indexed by strict partitions $\mu$ contained in the ``shifted staircase'' $(n-1, n-3, n-5,\dots)$  \cite[Corollary 5.22]{HMP4}.
By contrast, it is an open problem to 
identify a basis for 
$\QQ\spanning\{ \iF_z : z \in \tI_n \} \subset \Sym^{(n)}$.
Computer calculations indicate that
no subset of $\{ \iF_z : z \in \tI_n\}$ 
gives a positive basis for this space, that is, a basis in which every $\iF_z$ can be expressed with positive coefficients.
Thus, the question of how to define the ``Grassmannian'' elements of $\tI_n$
is subtler than for $\tS_n$.

We remark that the
\emph{type C Stanley symmetric functions} for signed permutations also generate the vector space spanned by the Schur $P$-functions. We do not know if there is a similar parallel between our affine involution Stanley symmetric functions and the \emph{affine type C Stanley symmetric functions} considered in \cite{LamSchillingShimozono}.

Finally, there are obvious ``left-handed'' versions of Propositions~\ref{coalg-prop} and \ref{comod-prop}.
These statements would suggest $\omega^+(\iF_z) = \sum_{\pi \in \cA(z)} F_{\pi^{-1}}$ instead of $\iF_z$
as the natural symmetric function corresponding to $z \in \tI_n$. 
Computations support the following conjecture, which implies that the choice of left- or right-handed 
convention is immaterial.

\begin{conjecture}\label{*+-conj} If $z \in \tI_n$ then $\omega^+(\iF_z) = \iF_{z}$, 
that is, $\sum_{\pi \in \cA(z)} \tF_{\pi^{-1}} = \sum_{\pi \in \cA(z)} \tF_{\pi}$.
\end{conjecture}

By Corollary~\ref{aut-cor}, this conjecture is equivalent to the claim that $\iF_z  =\iF_{z^*}$ for all $z \in \tI_n$.

\section{Transition formulas}\label{trans-sect}

Given two elements $\pi,\sigma \in \tS_n$,
write $\pi \lessdot \sigma$ if $\ell(\sigma) = \ell(\pi)+1$ and $\sigma = \pi t_{ij}$ 
for some $i<j \not \equiv i \modu n)$.
The transitive closure of $\lessdot$, denoted $\leq$, is the \emph{(strong) Bruhat order} on $\tS_n$.
The relation $\pi \lessdot \pi t_{ij}$ is equivalent to
the following more explicit condition:

\begin{lemma}[{\cite[Proposition 8.3.6]{BB}}]\label{lem-2}
Fix $\pi \in \tS_n$ and integers $i < j\not \equiv i \modu n)$.
One has $\pi \lessdot \pi t_{ij}$ if and only if  $\pi(i) < \pi(j)$ and no integer $e \in \ZZ$ has $i<e<j$ and $\pi(i) < \pi(e) < \pi(j)$.
\end{lemma}

For $\pi \in \tS_n$ and $ r \in \ZZ$ define the sets
\be\label{phi-sets-eq}
\ba
 \Psi^-_r(\pi) &= \{ \sigma \in\tS_n : \pi \lessdot \sigma = \pi t_{ir}\text{ for some integer }i < r\text{ with }i\notin r + n\ZZ\},
\\
\Psi^+_r(\pi) &= \{ \sigma\in\tS_n : \pi \lessdot \sigma= \pi t_{rj}\text{ for some integer }j > r\text{ with }j\notin r + n\ZZ\}
.
\ea
\ee
Lam and Shimozono \cite{LamShim} proved the following \emph{transition formula} for $\tF_\pi$:

\begin{theorem}[{\cite[Theorem 7]{LamShim}}]\label{trans-thm}
If $\pi \in \tS_n$ and $r \in \ZZ$ then
$\label{trans-eq} \sum_{\sigma \in \Psi^-_r(\pi)} \tF_\sigma = \sum_{\sigma \in \Psi^+_r(\pi)} \tF_\sigma$.
\end{theorem}

This result is an affine generalization of
the transition formula of Lascoux and Sch\"utzenberger \cite{Lascoux} for Schubert polynomials.
Lam and Shimozono originally hoped to use such identities to give a direct, algebraic proof of Theorem~\ref{+-thm}, 
but an argument along these lines remains to be found \cite[\S3.3]{LamShim}. Their transition formula has found other applications, however  \cite{Paw}.

\begin{example}
Suppose $n = 4$ and $\pi = [1,0,2,7] \in \tS_4$. 
 Setting $r=3$, we have 
\[ \ba 
\Psi^-_3(\pi) &= \{[2, 0, 1, 7], [1, 2, 0, 7]\} =\{ \pi t_{i,3} : i \in \{1,2\} \} ,  \\
\Psi^+_3(\pi) &= \{[1, 0, 7, 2], [-2, 0, 5, 7],  [1, -2, 4, 7]\} = \{ \pi t_{3,j} : j \in \{4,5,6\} \},
\ea
\]
and $\tF_{[2, 0, 1, 7]} + \tF_{[1, 2, 0, 7]}  = \tF_{[1, 0, 7, 2]}  +  \tF_{[-2, 0, 5, 7]} + \tF_{[1, -2, 4, 7]} 
=
\tF_{21^3} + \tF_{2^21} + \tF_{31^2} + \tF_{32}$.
\end{example}

The goal of this section is to prove an analogue of Theorem~\ref{trans-thm} for (affine) involution Stanley symmetric functions.
Continue to let $<$ denote the Bruhat order on $\tS_n$ and
write $\lessdot_I$ for the covering relation of $<$ restricted to $\tI_n$,
so that $y\lessdot_I z$ for $y,z \in \tI_n$ if and only if $\{ \pi \in \tI_n :  y \leq \pi <z \}  = \{y\}$.
For each pair of integers $ i<j\not\equiv i \modu n)$, there are associated operators
\[ \tau^n_{ij} : \tI_n \to \tI_n\]
that will play the role of multiplication by a reflection in the poset $(\tI_n,<)$.
Just as $\pi \lessdot \sigma$ only if $\sigma = \pi t_{ij}$ for some $i,j$,
it will hold that $y \lessdot_I z$ only if $z = \tau^n_{ij}(y)$ for some $i,j$.
The description of the maps $\tau^n_{ij}$
requires some auxiliary terminology.

Fix an involution $y \in \tI_n$ and integers $i < j \not \equiv i \modu n)$.
Define $\cG_{ij}(y)$ to be the graph 
with vertex set $ \{i,j,y(i),y(j)\}$
and edge set
$ \{ \{ i, y(i)\}, \{j, y(j)\} \} \setminus \{ \{i\}, \{j\}\},$
in which the vertices $i$ and $j$ are colored white and 
all other vertices are colored black.
Let $\sim$ be the equivalence relation on vertex-colored graphs with integer vertices
 in which $\cG \sim \cH$ if and only there exists a graph isomorphism $\cG \to \cH$
 that is order-preserving on vertex sets.
Finally, writing  $m \in \{2,3,4\}$ for the size of $\{  i, j, y(i), y(j) \}$, define
 $\cD_{ij}(y)$ to be 
the
unique 
vertex-colored graph on $[m]$  satisfying $\cD_{ij}(y) \sim \cG_{ij}(y)$.

There are twenty possibilities for $\cD_{ij}(y)$, which we draw by arranging the vertices  in order from left to right,
using $\circ$ and $\bullet$ for the white and black vertices.
For example,
if $y, z \in \tI_n$ are involutions such that $y(i) < j =y(j) < i$ and $i < z(j) < j < z(i)$, then 
\[ \cD_{ij}(y) = \diagramcBA \qquand \cD_{ij}(z) = \diagramDcBa. \]
Observe that if $i \equiv y(j) \modu n)$ then $\cD_{ij}(y)$ must be either 
\[
\diagrambADc \qquord 
\diagramBadC \qquord
\diagramCdaB\qquord
\diagramcDAb.
\]
In all other cases, $i \not\equiv y(j) \modu n)$.
The following slightly rephrases \cite[Definition 8.6]{M}:

\begin{definition}\label{tau-def}
Fix $y \in \tI_n$ and $i,j \in \ZZ$  with $i < j \not \equiv i \modu n)$. 
We write 
\[
t_{ii} = t_{jj} = 1,
\qquad
 (\circ, \circ) = t_{ij},
 \qquad
(\circ,\bullet) = t_{i,y(j)},
\qquand
 (\bullet,\circ) = t_{y(i),j}.
 \]
Let
 $ \overline y \in \tI_n$ be the affine permutation fixing each integer in the set $ \{i,j,y(i),y(j)\} + n \ZZ$
 and acting on all other integers as $k\mapsto y(k)$.
Finally, define $\tau^n_{ij}(y) \in \tI_n$ by
\[\tau^n_{ij}(y) =\begin{cases}
 (\circ, \circ)\cdot y\cdot (\circ, \circ) & \text{if $\cD_{ij}(y) $ is 
 $\diagrambADc$ or
 $\diagramCDab$ or 
 $\diagramcdAB$}
  \\
(\circ,\bullet)\cdot y\cdot (\circ,\bullet)  & \text{if $\cD_{ij}(y)$ is $\diagramCdaB$ and $i \not \equiv y(j) \modu n)$}
   \\
(\circ, \circ)\cdot \overline y  & \text{if $\cD_{ij}(y) $ is $\diagramCdaB$ and $i \equiv y(j) \modu n)$} 
\\
(\circ,\circ)\cdot  \overline y &\text{if $\cD_{ij}(y)$ is 
$\diagramAB$ or
$\diagramAcB$ or
$\diagramBaC$
or
$\diagramBadC$
}
\\
 (\circ,\bullet) \cdot \overline y&\text{if $\cD_{ij}(y)$ is  $\diagramACb$ or  $\diagramBaDc$}
\\
(\bullet,\circ)  \cdot \overline y&\text{if $\cD_{ij}(y)$ is $\diagrambAC$ or $\diagrambAdC$}
\\
y&\text{otherwise}.
\end{cases}
\]
\end{definition}

\begin{example}
Let $n=8$. Some examples of $\tau^n_{ij}(y)\neq y$ when $i\not \equiv y(j) \modu n)$:
\begin{itemize}
\item If $y = t_{1,4}t_{5,7}$ then $\tau^n_{4,5}(y) = t_{1,5}t_{4,7}$
and $\tau^n_{1,7}(y) = \tau^n_{1,5}(y) = \tau^n_{4,7}(y)= t_{1,7}$.

\item
If $y =  t_{1,5}t_{4,7}$ then $\tau^n_{1,7} (y)= t_{1,7}t_{4,5}$
and $\tau^n_{2,3}(y) = t_{1,5}t_{2,3}t_{4,7}$.

\end{itemize}
Some examples of $\tau^n_{ij}(y)\neq y$ when $i \equiv y(j) \modu n)$:
\begin{itemize}
\item
If $y = t_{1,8}$ then $\tau^n_{8,9}(y) = t_{8,17}$ and $\tau^n_{1,16}(y) = t_{1,16}$.

\item
If $y = t_{1,10}$ then $\tau^n_{1,18}(y) = t_{1,18}$.
\end{itemize}
\end{example}

The operators $\tau_{ij}^n$  are affine analogues of the ``covering transformations'' studied in \cite{HMP3,Incitti1}.
They are related to the Bruhat order on affine involutions by the following theorem.

\begin{theorem}[\cite{M}]\label{tau-thm2}
Suppose $y, z \in \tI_n$. The following properties are equivalent:
\ben
\item[(a)] It holds that $y\lessdot_I z$.
\item[(b)] For each $\sigma \in \cA(z)$, there exists an atom $\pi\in \cA(y)$ with $\pi \lessdot \sigma$.
\item[(c)] One has $\ellhat(z) =\ellhat(y) + 1$ and $z = \tau^n_{ij}(y)$ for some $i< j\not\equiv i \modu n)$.
\een
\end{theorem}

\begin{proof}
The equivalence of (a) and (b) holds by results in \cite{H1,H2}; see \cite[Proposition 8.1 and Lemma 8.2]{M}.
The equivalence of (a) of (c) holds by \cite[Corollary 8.12]{M}.
\end{proof}

One always has $y\leq \tau^n_{ij}(y)$ \cite[Lemma 8.8]{M}, but determining if $y \lessdot_I \tau^n_{ij}(y)$
can be complicated; see \cite[Proposition 8.9]{M}.
The following is often useful for this purpose:

 \begin{lemma}[\cite{M}]
 \label{length-lem}
 Suppose $y \in \tI_n$ and $i < j \not \equiv i \modu n)$ are such that $y \neq \tau^n_{ij}(y)$.
 Assume $i \not \equiv y (j) \modu n)$ and either $y(i) \leq i$ or $j \leq y(j)$.
 Then $y \lessdot_I \tau^n_{ij}(y)$ if and only if $y \lessdot yt_{ij}$.
\end{lemma}

\begin{proof}
This is the first half of \cite[Proposition 8.9(a)]{M}.
\end{proof}

The proof of our transition formula for $\iF_y$ relies on two technical results,
the first of which is the following theorem.

\begin{theorem}[Covering property] \label{tau-thm1}
Suppose $y,z \in \tI_n$ and $\pi \in \cA(y)$. Fix integers  $i<j \not \equiv i \modu n)$
 such that $\pi \lessdot \pi t_{ij}$.
Then $\pi t_{ij} \in \cA(z)$
if and only if $z =\tau^n_{ij}(y) \neq y$.
\end{theorem}

One half of this result is \cite[Theorem 8.10]{M}.
We delay giving the proof of the other half until Section~\ref{tau-sect1}.

As a shorthand in the following proposition, we write 
$\pi^{-1} = \dash a\dash b\dash c \dash \cdots\dash d \dash$
to mean that $\pi \in \tS_n$ has $\pi(a) < \pi(b) < \pi(c) < \dots < \pi(d)$.

\begin{proposition} \label{tech-prop}
Suppose $ y \in \tI_n$ and $\pi \in \cA(y)$.
Fix integers $i < j \not\equiv i \modu n)$ such that $\pi \lessdot \pi t_{ij}$
and  $\tau^n_{ij}(y) =y $.
One of the following cases then occurs:
\begin{enumerate}
\item[(A1)] $\cD_{ij}(y)$ is $\diagramCBa$ 
and $\pi^{-1} = \dash y(i) \dash i \dash j \dash$.
\item[(A2)] $\cD_{ij}(y)$ is $\diagramDCba$ 
and $\pi^{-1} = \dash y(j) \dash y(i) \dash i \dash j \dash$.
\item[(A3)] $\cD_{ij}(y)$ is $\diagramDcBa$ 
and $\pi^{-1} = \dash y(i) \dash i  \dash j  \dash y(j) \dash$.
\item[(B1)] $\cD_{ij}(y)$ is $\diagramcBA$ 
and $\pi^{-1} =  \dash i \dash j \dash y(j) \dash$.
\item[(B2)] $\cD_{ij}(y)$ is $\diagramdCbA$ 
and $\pi^{-1} = \dash y(i) \dash i  \dash j  \dash y(j) \dash$.
\item[(B3)] $\cD_{ij}(y)$ is $\diagramdcBA$ 
and $\pi^{-1} = \dash  i \dash j \dash y(j) \dash y(i) \dash$.
\item[(C1)] $\cD_{ij}(y)$ is $\diagramDCba$ 
and $\pi^{-1} = \dash y(i) \dash i \dash y(j) \dash j \dash$.
\item[(C2)] $\cD_{ij}(y)$ is $\diagramdcBA$ 
and $\pi^{-1} = \dash  i \dash y(i)  \dash j  \dash y(j)  \dash$.
\end{enumerate}
\end{proposition}

Observe that in each case we have $i \not\equiv y(j) \modu n)$.

\begin{proof}
The only way one can have $y = \tau^n_{ij}(y)$
outside the given cases is if
$y(j) = i < j = y(i)$ or $y(j) < i < j < y(i)$, but then
Theorem~\ref{local-thm} implies that $\pi \in \cA(y)$ has
$\pi(j) < \pi(i)$,
so it cannot hold that $\pi\lessdot \pi t_{ij}$.
When $y$ corresponds to the one of the given cases,
the possibilities for $\pi \in \cA(y)$ with $\pi\lessdot \pi t_{ij}$
are completely determined by Lemma~\ref{lem-2} and Theorem~\ref{local-thm}.
\end{proof}

This sets up the statement of our second technical theorem.

\begin{theorem}[Toggling property]\label{tau-thm3}
Suppose $y \in \tI_n$ and $\pi \in \cA(y)$.
Fix integers  $i<j \not \equiv i \modu n)$ such that $\pi \lessdot \pi t_{ij}$ and $y=\tau^n_{ij}(y)$.
Relative to the statement of Proposition~\ref{tech-prop}, define 
\[
k = \begin{cases}
j &\text{in cases (A1)-(A3)} \\
y(j) &\text{in cases (B1)-(B3), (C1), (C2)},
\end{cases}
\quad
l =  \begin{cases}
y(i) &\text{in cases (A1)-(A3), (C1), (C2)} \\
i &\text{in cases (B1)-(B3).}
\end{cases}
\]
Then $k<l\not \equiv k \modu n)$ and $\pi \neq \pi t_{ij} t_{kl} \in \cA(y)$ and $y=\tau^n_{kl}(y)$.
\end{theorem}

\begin{remark}
If the product 
$s_{i_1} \cdots s_{i_m}$ is not reduced 
but
$s_{i_1}\cdots \widehat{s_{i_j}} \cdots s_{i_m}$
is a reduced expression for some $\pi \in \tS_n$, where $\widehat{s_{i_j}}$
denotes the omission of one factor, 
then there exists a \emph{unique} index $j\neq k$ such that 
$s_{i_1}\cdots \widehat{s_{i_k}} \cdots s_{i_m}$ is also a reduced expression
for $\pi$ \cite[Lemma 21]{LamShim}. In fact, Lam and Shimozono prove in \cite{LamShim} that this property,
stated in slightly more general language, holds
for arbitrary Coxeter systems.

Fix $y \in \tI_n$.
Theorem~\ref{tau-thm3} implies that 
if
$s_{i_1} \cdots s_{i_m}$ is not a reduced word for any atom of any $z \in \tI_n$, 
and $s_{i_1}\cdots \widehat{s_{i_j}} \cdots s_{i_m}$ is a reduced expression
for some $\pi \in \cA(y)$, then there exists an index $j\neq k$
such that $s_{i_1}\cdots \widehat{s_{i_k}}\cdots s_{i_m}$
is a reduced expression for some (possibly different) atom $\sigma \in \cA(y)$.
We suspect but do not know how to prove that  $k$ is
again uniquely
determined; 
see \cite[Conjecture 3.42]{HMP3}.
This at least holds if $y \in I_n \subset \tI_n$ \cite[Lemma 3.34]{HMP3}.
\end{remark}

The proof of Theorem~\ref{tau-thm3} is at the end of Section~\ref{tau-sect1}.
We turn to some easier lemmas.

\begin{lemma}\label{ijkl-lem}
Let $y \in \tI_n$.
Suppose $i < j \not\equiv i \modu n)$ and $k < l \not \equiv k \modu n)$ are integers 
with $\tau^n_{ij}(y) = \tau^n_{kl}(y)\neq y$.
Then some $m \in \ZZ$ is such that 
 $k + mn \in \{ i, y(i)\}$ and $l + mn \in \{j, y(j)\}$.
\end{lemma}

\begin{proof}
This follows by inspecting Definition~\ref{tau-def}.
\end{proof}

Let $\Refl(\tS_n) = \{t_{ij} \in \tS_n : i<j\not\equiv i \modu n)\}$ 
be the set of reflections in $\tS_n$.  

\begin{lemma}\label{atom-bij-lem}
Suppose $y, z \in \tI_n$ and $y\lessdot_I z$.
The map $(\pi,t) \mapsto \pi t$ is a bijection
from the set of pairs $(\pi ,t) \in \cA(y) \times \Refl(\tS_n)$
with 
$\pi \lessdot \pi t $ and $ t=t_{ij}$ where $z = \tau^n_{ij}(y)$
 to the set of atoms  $\cA(z)$.
\end{lemma}

\begin{proof}
Theorems~\ref{tau-thm2} and \ref{tau-thm1} imply that the given map is surjective. 
To prove that the map is injective, 
fix integers $i<j \not\equiv i \modu n)$ and $k<l \not \equiv l \modu n)$ and 
suppose $\pi_1,\pi_2\in\cA(y)$ are such that 
$\pi_1t_{ij}=\pi_2t_{kl}\in\cA(z)$ and $z=\tau_{ij}^n(y)=\tau_{kl}^n(y)$.
 By Lemma~\ref{ijkl-lem}, we may assume that $k\in\{ i, y(i)\}$ and $l\in \{j, y(j)\}$.
 We must show that $t_{ij}=t_{kl}$ and $\pi_1=\pi_2$.
 There are four cases:
\begin{enumerate}
    \item[(1)] $i$ and $j$ are fixed points of $y$, in which case $i=k$ and $j=l$.
    \item[(2)] $i$ but not $j$ is a fixed point of $y$, in which case $i=k$.
    \item[(3)] $j$ but not $i$ is a fixed point of $y$, in which case $j=l$.
    \item[(4)] $i$ and $j$ are not fixed points of $y$.
\end{enumerate}
In case (1), it clearly holds that $t_{ij} =t_{kl}$ and $\pi_1=\pi_2$. 
Suppose we are in case (2).
If $j=l$ then we have 
$t_{ij} =t_{kl}$ and
$\pi_1=\pi_2$ as before. If $y(j) = l$ then $i<\min\{j,y(j)\}$ and $i\not\equiv y(j)\modu n)$, so Theorem~\ref{local-thm}
implies that $\pi_1(i) < \pi_1(y(j))$ and $\pi_2(i) < \pi_2(j)$, which leads to the contradiction
\[  \pi_2(i) =\pi_1t_{ij} t_{kl} (i) =\pi_1t_{ij}(y(j)) = \pi_1(y(j)) >\pi_1(i)= \pi_2 t_{kl}t_{ij}(i) = \pi_2t_{kl}(j) =\pi_2(j).\]
We deduce that $t_{kl}=t_{ij}$ and $\pi_1=\pi_2$ in case (2), as desired.
In case (3), the same conclusion follows by a symmetric argument.

Finally suppose we are in case (4).
Since $y \neq \tau^n_{ij}(y)$, consulting Definition~\ref{tau-def} shows that
we must have $\min\{i,y(i)\} < \min\{j,y(j)\}$ and $\max\{i,y(i)\} < \max\{j,y(j)\}$,
so
it follows from Theorem~\ref{local-thm}
that
$\pi_1(a) < \pi_1(b)$ and $\pi_2(a) < \pi_2(b)$ for all $a \in \{i,y(i) \} = \{k,y(k)\}$ and
$b \in \{j,y(j)\} = \{l,y(l)\}$.
If $i=k<l=y(j)$ or $y(i)=k<l=j$,  then it must hold that $i \not \equiv y(j) \modu n)$
and we derive the contradiction $\pi_2(i) >\pi_2(j)$ as in the previous cases.
If $y(i)=k<l=y(j)$ and $i \not \equiv y(j) \modu n)$ then we have the same contradiction
$\pi_2(i) = \pi_1(j) > \pi_1(i) = \pi_2(j)$.
The only way we can have $y(i) = k<l = y(j)$ and $i \equiv y(j) \modu n)$ 
is if $\max\{i,y(i)\} < \min\{j,y(j)\}$, but then $\tau^n_{ij}(y) \neq \tau^n_{kl}(y)$.
The only remaining possibility is to have $t_{ij} =t_{kl}$, in which case $\pi_1=\pi_2$ as desired.
\end{proof}

\def\iPhi{\Phi}

Fix $y \in \tI_n$. For $ r \in \ZZ$, define
\be\label{iphi-sets-eq}
\ba
 \iPhi^-_r(y) &= \left\{ z \in \tI_n : y \lessdot_I z=\tau^n_{ir}(y) \text{ for some }i < r\text{ with }i\notin \{r,y(r)\} + n\ZZ\right\},
\\
\iPhi^+_r(y) &= \left\{ z \in \tI_n : y \lessdot_I z=\tau^n_{rj}(y)\text{ for some }j > r\text{ with }j\notin \{r,y(r)\} + n\ZZ\right\}
.
\ea
\ee
In the cases of primary interest, these analogues of \eqref{phi-sets-eq} have a  simpler description.

\begin{lemma}\label{phi-lem}
Suppose $y \in \tI_n$ and $(p,q) \in \Cyc(y) := \{ (i,j) \in \ZZ \times \ZZ : i \leq j=y(i)\}$.
Then:
\ben
\item[(a)] $\iPhi^-_p(y) =  \left\{ \tau^n_{ip}(y) : i<p\text{ and }i \notin \{p,q\} + n \ZZ\text{ and }y \lessdot yt_{ip} \right\}$.
\item[(b)] $\iPhi^+_q(y) =  \left\{ \tau^n_{qj}(y) : j>q\text{ and }j \notin \{p,q\} + n \ZZ\text{ and }y \lessdot yt_{qj} \right\}$.
\item[(c)] $\iPhi^-_p(y) \supset \iPhi^-_q(y)$ and $\iPhi^+_p(y) \subset \iPhi^+_q(y)$ and  $\iPhi^-_p(y)\cap \iPhi^+_q(y) = \varnothing$.
\een
\end{lemma}

\begin{proof}
Parts
(a) and (b) are immediate corollaries of Theorem~\ref{length-lem}. 
To show that $\iPhi^-_p(y) \supset \iPhi^-_q(y)$, let $(p,q)\in\Cyc(y)$. 
It suffices to check that 
if
$j<q$ and $j \notin \{p,q\} +n\ZZ$ and $y \neq \tau^n_{jq}(y)$
then there exists
$i < p$ with $i \notin \{p,q\} + n \ZZ$ and $\tau^n_{ip}(y) =\tau^n_{jq}(y)$.
This is straightforward from Definition~\ref{tau-def}.
The proof that $\iPhi^+_p(y) \subset \iPhi^+_q(y)$ is similar.
Finally, the sets $\iPhi^-_p(y)$ and $ \iPhi^+_q(y) $ are disjoint because
we can only have $\tau^n_{ip}(y) = \tau^n_{qj}(y)$ if 
$i,j \in \{p,q\} + n\ZZ$ by Lemma~\ref{ijkl-lem}.
\end{proof}

We now present our transition formula for $\iF_y$.
This is both an involution analogue of Theorem~\ref{trans-thm}
and an affine generalization of \cite[Theorem 3.10]{HMP4}.
The latter result is itself the ``stable limit'' of a transition formula
for \emph{involution Schubert polynomials} \cite[Theorem 1.5]{HMP3}.

\begin{theorem}\label{main-thm}
 Suppose $y \in \tI_n$ and $p,q \in \ZZ$ are such that $p\leq q = y(p)$. Then 
\[\sum_{z \in \iPhi_p^-(y)} \iF_z = \sum_{z \in \iPhi^+_q(y)} \iF_z.\]
\end{theorem}

\def\LHS{L}
\def\RHS{R}

\begin{proof}
If $i<p\not\equiv i \modu n)$ and $i \in q + n \ZZ$, then
$t_{ip} = t_{qj}$ where $q < j = p + q - i \in p + n\ZZ$.
Likewise, if $i<q\not\equiv i \modu n)$ and $i \in p + n \ZZ$, then
$t_{iq} = t_{pj}$ where $p<j = p + q - i \in p + n\ZZ$.
In view of these observations,
Lam and Shimozono's transition formula (Theorem~\ref{trans-thm}) implies that 
\be\label{id1-eq}
\sum_{\pi \in \cA(y)}\(\sum_{\substack{i < p   \\ \pi \lessdot \pi t_{ip}}} \tF_{\pi t_{ip}}
+
\sum_{\substack{i < q  \\ \pi \lessdot \pi t_{iq}}} \tF_{\pi t_{iq}}\)
=
\sum_{\pi \in \cA(y)}
\(
\sum_{\substack{p<i  \\ \pi \lessdot \pi t_{pi}}} \tF_{\pi t_{pi}}
+
\sum_{\substack{q < i   \\ \pi \lessdot \pi t_{qi}}} \tF_{\pi t_{qi}}
\)
\ee
where the inner sums range over integers $i \notin \{p,q\} + n\ZZ$.

We examine the left side of \eqref{id1-eq}.
Suppose $\pi \in \cA(y)$ and $i,j \notin \{p,q\}+n\ZZ$ are such that $i<p$ and $j<q$ and $\pi \lessdot \pi t_{ip}$ and $\pi \lessdot \pi t_{jq}$.
Theorem~\ref{tau-thm1} implies that
$\pi t_{ip}$ (respectively, $\pi t_{jq}$) is an atom for some $z \in \tI_n$ if and only if
 $y \neq z = \tau^n_{ip}(y)$ (respectively, $y \neq z= \tau^n_{jq}(y)$), in which case 
 $z \in \iPhi^-_p(y)$ by Theorem~\ref{tau-thm2} and Lemma~\ref{phi-lem}(c).
 Conversely, if $z \in \iPhi^-_p(y)$, then it follows by 
Lemmas~\ref{ijkl-lem} and \ref{atom-bij-lem} that there exists a unique integer $i \notin \{p,q\} + n\ZZ$ with either $i < p$ and $\pi\lessdot \pi t_{ip}  \in \cA(z)$
or $i < q$ and $\pi \lessdot \pi t_{iq} \in \cA(z)$.
We conclude that   the left side of \eqref{id1-eq} can be rewritten as
$
\sum_{z \in \iPhi^-_p(y)} \iF_z + \sum_{(\pi,i,j) \in \cN^-} \tF_{\pi t_{ij}}
$
where 
$\cN^-$ is the set of triples $(\pi,i,j) \in \cA(y) \times \ZZ \times \ZZ$ 
with $i < j \in \{p,q\}$ and $i \notin \{p,q\} + n \ZZ$ and
$\pi\lessdot \pi t_{ij}$ and $\tau^n_{ij}(y) = y$.
By a symmetric argument, the right side of \eqref{id1-eq} is  
$
\sum_{z \in \iPhi^+_q(y)} \iF_z + 
 \sum_{(\pi,i,j) \in \cN^+} \tF_{\pi t_{ij}}
$
where $\cN^+$ is the set of triples $(\pi,i,j) \in \cA(y) \times \ZZ \times \ZZ$ 
with $j>i\in \{p,q\}$ and $j \notin \{p,q\} + n \ZZ$ and
$\pi\lessdot \pi t_{ij}$ and $\tau^n_{ij}(y) = y$.
It therefore suffices to show that $\sum_{(\pi,i,j) \in \cN^-} \tF_{\pi t_{ij}} = \sum_{(\pi,i,j) \in \cN^+} \tF_{\pi t_{ij}}$.
In fact, we will show that 
these sums involve the same set of Stanley symmetric functions $\tF_\sigma$.

Let 
$\cN \supset \cN^\pm$ be the set of triples $(\pi,i,j) \in \cA(y)\times \ZZ \times \ZZ$
with
 $i< j \not\equiv i \modu n)$ and
$\pi \lessdot \pi t_{ij}$ and
 $\tau^n_{ij}(y) =y$.
 Given $(\pi,i,j) \in \cN$, let 
 $k<l \not \equiv  k\modu n)$ be as in Theorem~\ref{tau-thm3}
 so that $ \pi \neq \pi t_{ij} t_{kl} \in \cA(y)$
 and $\tau^n_{kl}(y)=y$,
 and let
  $\theta(\pi,i,j) = (\pi t_{ij} t_{kl},k,l).$
This evidently defines a map $\theta : \cN \to \cN$.
It is a straightforward but tedious exercise to
work through the cases in Proposition~\ref{tech-prop} to check
 that 
$\theta$ is actually an involution; the details are left to the reader.
Given this property, it follows from Theorem~\ref{tau-thm3}
that $\theta$ restricts to a bijection $\cN^- \to \cN^+$.
Since  $\theta(\pi_1,i,j) = (\pi_2,k,l)$ implies that $\pi_1 t_{ij} = \pi_2 t_{kl}$,
we deduce that $\sum_{(\pi,i,j) \in \cN^-} \tF_{\pi t_{ij}} = \sum_{(\pi,i,j) \in \cN^+} \tF_{\pi t_{ij}}$
as needed.
\end{proof}

 \begin{example}
Suppose $n=4$ and 
\[y =
\begin{tikzpicture}[baseline=0,scale=0.15,label/.style={postaction={ decorate,transform shape,decoration={ markings, mark=at position .5 with \node #1;}}}]
{
\draw[fill,lightgray] (0,0) circle (4.0);
\node at (2.4492935982947064e-16, 4.0) {$_\bullet$};
\node at (1.8369701987210297e-16, 3.0) {$_{1}$};
\node at (4.0, 0.0) {$_\bullet$};
\node at (3.0, 0.0) {$_{2}$};
\node at (2.4492935982947064e-16, -4.0) {$_\bullet$};
\node at (1.8369701987210297e-16, -3.0) {$_{3}$};
\node at (-4.0, -4.898587196589413e-16) {$_\bullet$};
\node at (-3.0, -3.6739403974420594e-16) {$_{4}$};
\draw [-,>=latex,domain=0:100,samples=100] plot ({(4.0 + 2.0 * sin(180 * (0.5 + asin(-0.9 + 1.8 * (\x / 100)) / asin(0.9) / 2))) * cos(90 - (180.0 + \x * 4.5))}, {(4.0 + 2.0 * sin(180 * (0.5 + asin(-0.9 + 1.8 * (\x / 100)) / asin(0.9) / 2))) * sin(90 - (180.0 + \x * 4.5))});
}
\end{tikzpicture}
= t_{3,8} = [1, 2, 8, -1] \in \tI_4.
\]
Setting $p=q=2$, we have
\[ \ba 
\iPhi^-_2(y) &= \left\{
\begin{tikzpicture}[baseline=0,scale=0.15,label/.style={postaction={ decorate,transform shape,decoration={ markings, mark=at position .5 with \node #1;}}}]
{
\draw[fill,lightgray] (0,0) circle (4.0);
\node at (2.4492935982947064e-16, 4.0) {$_\bullet$};
\node at (1.8369701987210297e-16, 3.0) {$_{1}$};
\node at (4.0, 0.0) {$_\bullet$};
\node at (3.0, 0.0) {$_{2}$};
\node at (2.4492935982947064e-16, -4.0) {$_\bullet$};
\node at (1.8369701987210297e-16, -3.0) {$_{3}$};
\node at (-4.0, -4.898587196589413e-16) {$_\bullet$};
\node at (-3.0, -3.6739403974420594e-16) {$_{4}$};
\draw [-,>=latex,domain=0:100,samples=100] plot ({(4.0 + 0.0 * sin(180 * (0.5 + asin(-0.9 + 1.8 * (\x / 100)) / asin(0.9) / 2))) * cos(90 - (0.0 + \x * 0.9))}, {(4.0 + 0.0 * sin(180 * (0.5 + asin(-0.9 + 1.8 * (\x / 100)) / asin(0.9) / 2))) * sin(90 - (0.0 + \x * 0.9))});
\draw [-,>=latex,domain=0:100,samples=100] plot ({(4.0 + 2.0 * sin(180 * (0.5 + asin(-0.9 + 1.8 * (\x / 100)) / asin(0.9) / 2))) * cos(90 - (180.0 + \x * 4.5))}, {(4.0 + 2.0 * sin(180 * (0.5 + asin(-0.9 + 1.8 * (\x / 100)) / asin(0.9) / 2))) * sin(90 - (180.0 + \x * 4.5))});
}
\end{tikzpicture}
\right\} 
=\{t_{1,2} t_{3,8}\} 
=\left\{ \tau^4_{1,2}(y)  \right\} ,  
\\
\iPhi^+_2(y) &= 
\left\{
\begin{tikzpicture}[baseline=0,scale=0.15,label/.style={postaction={ decorate,transform shape,decoration={ markings, mark=at position .5 with \node #1;}}}]
{
\draw[fill,lightgray] (0,0) circle (4.0);
\node at (2.4492935982947064e-16, 4.0) {$_\bullet$};
\node at (1.8369701987210297e-16, 3.0) {$_{1}$};
\node at (4.0, 0.0) {$_\bullet$};
\node at (3.0, 0.0) {$_{2}$};
\node at (2.4492935982947064e-16, -4.0) {$_\bullet$};
\node at (1.8369701987210297e-16, -3.0) {$_{3}$};
\node at (-4.0, -4.898587196589413e-16) {$_\bullet$};
\node at (-3.0, -3.6739403974420594e-16) {$_{4}$};
\draw [-,>=latex,domain=0:100,samples=100] plot ({(4.0 + 2.0 * sin(180 * (0.5 + asin(-0.9 + 1.8 * (\x / 100)) / asin(0.9) / 2))) * cos(90 - (90.0 + \x * 5.4))}, {(4.0 + 2.0 * sin(180 * (0.5 + asin(-0.9 + 1.8 * (\x / 100)) / asin(0.9) / 2))) * sin(90 - (90.0 + \x * 5.4))});
}
\end{tikzpicture},\
\begin{tikzpicture}[baseline=0,scale=0.15,label/.style={postaction={ decorate,transform shape,decoration={ markings, mark=at position .5 with \node #1;}}}]
{
\draw[fill,lightgray] (0,0) circle (4.0);
\node at (2.4492935982947064e-16, 4.0) {$_\bullet$};
\node at (1.8369701987210297e-16, 3.0) {$_{1}$};
\node at (4.0, 0.0) {$_\bullet$};
\node at (3.0, 0.0) {$_{2}$};
\node at (2.4492935982947064e-16, -4.0) {$_\bullet$};
\node at (1.8369701987210297e-16, -3.0) {$_{3}$};
\node at (-4.0, -4.898587196589413e-16) {$_\bullet$};
\node at (-3.0, -3.6739403974420594e-16) {$_{4}$};
\draw [-,>=latex,domain=0:100,samples=100,densely dotted] plot ({(4.0 + 2.0 * sin(180 * (0.5 + asin(-0.9 + 1.8 * (\x / 100)) / asin(0.9) / 2))) * cos(90 - (90.0 + \x * 2.7))}, {(4.0 + 2.0 * sin(180 * (0.5 + asin(-0.9 + 1.8 * (\x / 100)) / asin(0.9) / 2))) * sin(90 - (90.0 + \x * 2.7))});
\draw [-,>=latex,domain=0:100,samples=100] plot ({(4.0 + 4.0 * sin(180 * (0.5 + asin(-0.9 + 1.8 * (\x / 100)) / asin(0.9) / 2))) * cos(90 - (180.0 + \x * 4.5))}, {(4.0 + 4.0 * sin(180 * (0.5 + asin(-0.9 + 1.8 * (\x / 100)) / asin(0.9) / 2))) * sin(90 - (180.0 + \x * 4.5))});
}
\end{tikzpicture}
\right\}
=\{t_{2,8},\ t_{2,5} t_{3,8}\} 
= \left\{ \tau^4_{2,3}(y),\ \tau^4_{2,5}(y)\right \},
\ea
\]
and 
$\iF_{[2, 1, 8, -1]}   = \iF_{[1, 8, 3, -2]}  +  \iF_{[-2, 5, 8, -1]}  
=
\tF_{1^5} + \tF_{21^3} + \tF_{2^2 1} + \tF_{3 1^2} + \tF_{3 2}
.$
 \end{example}

 \begin{example}
Suppose $n=5$ and 
\[y =
\begin{tikzpicture}[baseline=0,scale=0.15,label/.style={postaction={ decorate,transform shape,decoration={ markings, mark=at position .5 with \node #1;}}}]
{
\draw[fill,lightgray] (0,0) circle (4.0);
\node at (2.4492935982947064e-16, 4.0) {$_\bullet$};
\node at (1.8369701987210297e-16, 3.0) {$_{1}$};
\node at (3.804226065180614, 1.2360679774997896) {$_\bullet$};
\node at (2.8531695488854605, 0.9270509831248421) {$_{2}$};
\node at (2.351141009169893, -3.2360679774997894) {$_\bullet$};
\node at (1.7633557568774196, -2.427050983124842) {$_{3}$};
\node at (-2.351141009169892, -3.23606797749979) {$_\bullet$};
\node at (-1.7633557568774192, -2.4270509831248424) {$_{4}$};
\node at (-3.8042260651806146, 1.2360679774997891) {$_\bullet$};
\node at (-2.853169548885461, 0.9270509831248419) {$_{5}$};
\draw [-,>=latex,domain=0:100,samples=100,densely dotted] plot ({(4.0 + 4.0 * sin(180 * (0.5 + asin(-0.9 + 1.8 * (\x / 100)) / asin(0.9) / 2))) * cos(90 - (72.0 + \x * 4.32))}, {(4.0 + 4.0 * sin(180 * (0.5 + asin(-0.9 + 1.8 * (\x / 100)) / asin(0.9) / 2))) * sin(90 - (72.0 + \x * 4.32))});
\draw [-,>=latex,domain=0:100,samples=100] plot ({(4.0 + 4.0 * sin(180 * (0.5 + asin(-0.9 + 1.8 * (\x / 100)) / asin(0.9) / 2))) * cos(90 - (216.0 + \x * 4.32))}, {(4.0 + 4.0 * sin(180 * (0.5 + asin(-0.9 + 1.8 * (\x / 100)) / asin(0.9) / 2))) * sin(90 - (216.0 + \x * 4.32))});
}
\end{tikzpicture}
= t_{2,8}t_{4,10} = [1, 8, -3, 10, -1] \in \tI_5.
\]
Setting $(p,q)=(2,8)$, we have
\[ \ba 
\iPhi^-_2(y) &= \left\{
\begin{tikzpicture}[baseline=0,scale=0.15,label/.style={postaction={ decorate,transform shape,decoration={ markings, mark=at position .5 with \node #1;}}}]
{
\draw[fill,lightgray] (0,0) circle (4.0);
\node at (2.4492935982947064e-16, 4.0) {$_\bullet$};
\node at (1.8369701987210297e-16, 3.0) {$_{1}$};
\node at (3.804226065180614, 1.2360679774997896) {$_\bullet$};
\node at (2.8531695488854605, 0.9270509831248421) {$_{2}$};
\node at (2.351141009169893, -3.2360679774997894) {$_\bullet$};
\node at (1.7633557568774196, -2.427050983124842) {$_{3}$};
\node at (-2.351141009169892, -3.23606797749979) {$_\bullet$};
\node at (-1.7633557568774192, -2.4270509831248424) {$_{4}$};
\node at (-3.8042260651806146, 1.2360679774997891) {$_\bullet$};
\node at (-2.853169548885461, 0.9270509831248419) {$_{5}$};
\draw [-,>=latex,domain=0:100,samples=100,densely dotted] plot ({(4.0 + 2.0 * sin(180 * (0.5 + asin(-0.9 + 1.8 * (\x / 100)) / asin(0.9) / 2))) * cos(90 - (72.0 + \x * 2.16))}, {(4.0 + 2.0 * sin(180 * (0.5 + asin(-0.9 + 1.8 * (\x / 100)) / asin(0.9) / 2))) * sin(90 - (72.0 + \x * 2.16))});
\draw [-,>=latex,domain=0:100,samples=100] plot ({(4.0 + 4.0 * sin(180 * (0.5 + asin(-0.9 + 1.8 * (\x / 100)) / asin(0.9) / 2))) * cos(90 - (216.0 + \x * 6.48))}, {(4.0 + 4.0 * sin(180 * (0.5 + asin(-0.9 + 1.8 * (\x / 100)) / asin(0.9) / 2))) * sin(90 - (216.0 + \x * 6.48))});
}
\end{tikzpicture},
\
\begin{tikzpicture}[baseline=0,scale=0.15,label/.style={postaction={ decorate,transform shape,decoration={ markings, mark=at position .5 with \node #1;}}}]
{
\draw[fill,lightgray] (0,0) circle (4.0);
\node at (2.4492935982947064e-16, 4.0) {$_\bullet$};
\node at (1.8369701987210297e-16, 3.0) {$_{1}$};
\node at (3.804226065180614, 1.2360679774997896) {$_\bullet$};
\node at (2.8531695488854605, 0.9270509831248421) {$_{2}$};
\node at (2.351141009169893, -3.2360679774997894) {$_\bullet$};
\node at (1.7633557568774196, -2.427050983124842) {$_{3}$};
\node at (-2.351141009169892, -3.23606797749979) {$_\bullet$};
\node at (-1.7633557568774192, -2.4270509831248424) {$_{4}$};
\node at (-3.8042260651806146, 1.2360679774997891) {$_\bullet$};
\node at (-2.853169548885461, 0.9270509831248419) {$_{5}$};
\draw [-,>=latex,domain=0:100,samples=100] plot ({(4.0 + 4.0 * sin(180 * (0.5 + asin(-0.9 + 1.8 * (\x / 100)) / asin(0.9) / 2))) * cos(90 - (0.0 + \x * 5.04))}, {(4.0 + 4.0 * sin(180 * (0.5 + asin(-0.9 + 1.8 * (\x / 100)) / asin(0.9) / 2))) * sin(90 - (0.0 + \x * 5.04))});
\draw [-,>=latex,domain=0:100,samples=100,densely dotted] plot ({(4.0 + 2.0 * sin(180 * (0.5 + asin(-0.9 + 1.8 * (\x / 100)) / asin(0.9) / 2))) * cos(90 - (216.0 + \x * 4.32))}, {(4.0 + 2.0 * sin(180 * (0.5 + asin(-0.9 + 1.8 * (\x / 100)) / asin(0.9) / 2))) * sin(90 - (216.0 + \x * 4.32))});
}
\end{tikzpicture}
\right\} 
=\{t_{2,5} t_{4,13},\ t_{1,8} t_{4,10}\} 
=\left\{ \tau^5_{-1,2}(y),\ \tau^5_{1,2}(y) \right\} ,  
\\
\iPhi^+_8(y) &= 
\left\{
\begin{tikzpicture}[baseline=0,scale=0.15,label/.style={postaction={ decorate,transform shape,decoration={ markings, mark=at position .5 with \node #1;}}}]
{
\draw[fill,lightgray] (0,0) circle (4.0);
\node at (2.4492935982947064e-16, 4.0) {$_\bullet$};
\node at (1.8369701987210297e-16, 3.0) {$_{1}$};
\node at (3.804226065180614, 1.2360679774997896) {$_\bullet$};
\node at (2.8531695488854605, 0.9270509831248421) {$_{2}$};
\node at (2.351141009169893, -3.2360679774997894) {$_\bullet$};
\node at (1.7633557568774196, -2.427050983124842) {$_{3}$};
\node at (-2.351141009169892, -3.23606797749979) {$_\bullet$};
\node at (-1.7633557568774192, -2.4270509831248424) {$_{4}$};
\node at (-3.8042260651806146, 1.2360679774997891) {$_\bullet$};
\node at (-2.853169548885461, 0.9270509831248419) {$_{5}$};
\draw [-,>=latex,domain=0:100,samples=100,densely dotted] plot ({(4.0 + 4.0 * sin(180 * (0.5 + asin(-0.9 + 1.8 * (\x / 100)) / asin(0.9) / 2))) * cos(90 - (72.0 + \x * 5.04))}, {(4.0 + 4.0 * sin(180 * (0.5 + asin(-0.9 + 1.8 * (\x / 100)) / asin(0.9) / 2))) * sin(90 - (72.0 + \x * 5.04))});
\draw [-,>=latex,domain=0:100,samples=100] plot ({(4.0 + 4.0 * sin(180 * (0.5 + asin(-0.9 + 1.8 * (\x / 100)) / asin(0.9) / 2))) * cos(90 - (144.0 + \x * 5.04))}, {(4.0 + 4.0 * sin(180 * (0.5 + asin(-0.9 + 1.8 * (\x / 100)) / asin(0.9) / 2))) * sin(90 - (144.0 + \x * 5.04))});
}
\end{tikzpicture},
\
\begin{tikzpicture}[baseline=0,scale=0.15,label/.style={postaction={ decorate,transform shape,decoration={ markings, mark=at position .5 with \node #1;}}}]
{
\draw[fill,lightgray] (0,0) circle (4.0);
\node at (2.4492935982947064e-16, 4.0) {$_\bullet$};
\node at (1.8369701987210297e-16, 3.0) {$_{1}$};
\node at (3.804226065180614, 1.2360679774997896) {$_\bullet$};
\node at (2.8531695488854605, 0.9270509831248421) {$_{2}$};
\node at (2.351141009169893, -3.2360679774997894) {$_\bullet$};
\node at (1.7633557568774196, -2.427050983124842) {$_{3}$};
\node at (-2.351141009169892, -3.23606797749979) {$_\bullet$};
\node at (-1.7633557568774192, -2.4270509831248424) {$_{4}$};
\node at (-3.8042260651806146, 1.2360679774997891) {$_\bullet$};
\node at (-2.853169548885461, 0.9270509831248419) {$_{5}$};
\draw [-,>=latex,domain=0:100,samples=100] plot ({(4.0 + 4.0 * sin(180 * (0.5 + asin(-0.9 + 1.8 * (\x / 100)) / asin(0.9) / 2))) * cos(90 - (72.0 + \x * 5.76))}, {(4.0 + 4.0 * sin(180 * (0.5 + asin(-0.9 + 1.8 * (\x / 100)) / asin(0.9) / 2))) * sin(90 - (72.0 + \x * 5.76))});
\draw [-,>=latex,domain=0:100,samples=100,densely dotted] plot ({(4.0 + 2.0 * sin(180 * (0.5 + asin(-0.9 + 1.8 * (\x / 100)) / asin(0.9) / 2))) * cos(90 - (216.0 + \x * 2.88))}, {(4.0 + 2.0 * sin(180 * (0.5 + asin(-0.9 + 1.8 * (\x / 100)) / asin(0.9) / 2))) * sin(90 - (216.0 + \x * 2.88))});
}
\end{tikzpicture}
\right\}
=\{t_{2,9} t_{3,10},\ t_{2,10} t_{4,8}\} 
= \left\{ \tau^5_{8,9}(y),\  \tau^5_{8,10}(y)\right \},
\ea
\]
and 
$\iF_{[1,5,-6,13,2]} +\iF_{[8,2,-4,10,-1]}   = \iF_{[1,9,10,-3,-2]}  +  \iF_{[1,10,-1,8,-3]}  
=
\tF_{21^7} + \tF_{2^21^5} + \tF_{2^31^3} + 2\tF_{2^41} + \tF_{31^6} + \tF_{321^4} + 3\tF_{32^21^2} + \tF_{32^3} + \tF_{3^21^3} + 2\tF_{3^221} + \tF_{3^3} + \tF_{421^3} + \tF_{42^21} + \tF_{431^2} + \tF_{432}
.$
 \end{example}

Let $\cR(\pi)$ denote the set of reduced expressions for $\pi \in \tS_n$,
and define $\iR(z) = \bigsqcup_{\pi \in \cA(z)} \cR(\pi)$
for $z \in \tI_n$.
Elements of $\iR(z)$ are called \emph{involution words} in \cite{HMP1,HMP2}. The same sequences, 
read in reverse order, are referred to as \emph{reduced $\underline S$-expressions} in \cite{HH,H2} and \emph{reduced $I_*$-expressions} in \cite{HuZhang,KL1}.

\begin{corollary}
 If $y \in \tI_n$, $p,q \in \ZZ$, and $p\leq q = y(p)$, then 
$\sum_{z \in \iPhi^-_p(y)} |\iR(z)| = \sum_{z \in \iPhi^+_q(y)} |\iR(z)|.$
\end{corollary}

\begin{proof}
The size of $\cR(\pi)$ is the coefficient of each square-free monomial in $\tF_\pi$,
so $|\iR(z)|$ is the coefficient of each square-free monomial in $\iF_z$,
and the corollary follows from Theorem~\ref{main-thm}.
\end{proof}

\begin{remark}
It is possible to give a bijective proof of the preceding identity
using an affine generalization of the ``involution Little map'' 
described in \cite[\S3.3]{HMP3}. 
We omit this material since the arguments for the affine case
are essentially unchanged, except that one substitutes
Lam and Shimozono's affine bumping algorithm \cite{LamShim}
for the classical Little map in a few obvious places.
In addition, since we do not yet have a good analogue of wiring diagrams for 
the elements of $\iR(z)$,
 the involution Little map is unsatisfyingly nonconstructive.
 Finding a more efficient way of representing involution words and
 computing the involution Little map is an open problem of interest.
\end{remark}

\section{Proof of covering and toggling properties}\label{tau-sect1}

This section contains the proofs of our main technical results, Theorems~\ref{tau-thm1}
and \ref{tau-thm3}.

\subsection{Self-contained arguments}

We split the proof of Theorem~\ref{tau-thm1}
across four lemmas in this and the next subsection.
In each lemma, we adopt the following hypothesis:

\begin{hypothesis}\label{lem-hyp}
Let 
$y \in \tI_n$ and $\pi \in \cA(y)$,
fix integers  $i<j \not \equiv i \modu n)$
and define $z = \tau^n_{ij}(y)$,
and assume that $\pi \lessdot \pi t_{ij}$ and $y\neq z$.
\end{hypothesis}

\begin{lemma}\label{tau-thm-lem1}
Assume the conditions in Hypothesis~\ref{lem-hyp}.
In addition, suppose that $i \equiv  y(j) \modu n)$ and $y(i) < i < j < y(j)$. Then $\pi t_{ij} \in \cA(z)$.
\end{lemma}

\begin{proof}
Let $i_0 = i$ and $j_0 = y(i)$ and define $i_m = i_0 + mn$ and $j_m =j_0 + mn$ for $m \in \ZZ$.
Let $A = \{ (j_m, i_m) : m \in \ZZ\}$, $B = \{ (i_m,j_{m+2}) : m \in \ZZ\}$, and $C = \{ (a,b) \in \Cyc(y) : a,b \notin \{i,j\} + n\ZZ\}$.
Then $\Cyc(y) = A \sqcup C$ and $\Cyc(z) = B \sqcup C$.
Theorem~\ref{local-thm} implies that 
\[\pi(i_m) < \pi(j_m) < \pi(i_{m+1}) < \pi(j_{m+1})\qquad\text{for all $m \in \ZZ$.}\]
Let $t=t_{ij}$.
Since $\pi \lessdot \pi t$,
Lemma~\ref{lem-2} implies that $j = j_1$ and $y(j) = i_1$. We have
\[\pi t(j_{m+1}) < \pi t(i_{m-1}) < \pi t(j_{m+2}) < \pi t(i_m)\qquad\text{for all $m \in \ZZ$}\]
and $\pi(a) = \pi t(a)$ and $y(a) = z(a)$ for all integers $a \notin \{i,j\} +n \ZZ$.
The conditions in Theorem~\ref{local-thm} therefore hold for $\pi t$ relative to $z$
for all cycles $(a,b),(a',b') \in B$ and all cycles $(a,b),(a',b') \in C$.
Hence, to show that $\pi t \in \cA(z)$, it suffices by Theorem~\ref{local-thm} to check 
for all $(a,b) \in C$ that:
\ben
\item[(1)] If $i <a \leq b < j+n$ then we do not have $\pi(i+n) < \pi(a) < \pi(j)$ or $\pi(i+n) < \pi(b) < \pi(j)$.

\item[(2)] If $a < i < j+n < b$ then we do not have $\pi(b) < \pi(i+n) < \pi(a) $ or $\pi(b) < \pi(j) < \pi(a)$.

\item[(3)] If $i < a$ and $j+n< b$ then $\pi(j) < \pi(b)$.

\item[(4)] If $a < i$ and $b < j+n$ then $\pi(a) < \pi(i+n)$.
\een
Since $(j-n,i),(j,i+n) \in \Cyc(y)$ and $\pi \in \cA(y)$,
Theorem~\ref{local-thm} has the following implications:
\begin{itemize}
    \item If $i<a\leq b<j<i+n<j+n$ then $\pi(i)<\pi(b)\leq \pi(a)<\pi(j)$.
    \item If $i<a<j<b<i+n<j+n$ then  $\pi(i)<\pi(b)<\pi(a)<\pi(j)$.
    \item If $i<a<j<i+n<b<j+n$ then either $\pi(i)<\pi(a)<\pi(j)$ or $\pi(i+n)<\pi(b)<\pi(j+n)$.
    \item If $i<j<a\leq b<i+n<j+n$ then the conditions in (1) must hold.
    \item If $i<j<a<i+n<b < j+n$ then $\pi(i+n)<\pi(b)< \pi(a)<\pi(j+n)$.
    \item If $i<j<i+n<a\leq b<j+n$ then $\pi(i+n)<\pi(b)\leq \pi(a)<\pi(j+n)$.
\end{itemize}
In the fourth case the desired conditions hold, while all of the other cases contradict Lemma~\ref{lem-2}
since $\pi \lessdot \pi t_{ij}$. We conclude that property (1) holds.

Property (2) follows from Theorem~\ref{local-thm} since $\pi \in \cA(y)$ and
if $a<i<j+n<b$ then $a<j<i+n<b$ and $y(j)=i+n$.
To check property (3), suppose $i<a$ and $j+n<b$. Theorem~\ref{local-thm} implies that 
$\pi(i)<\pi(b)\leq\pi(a)$, that if $j<a$ then $\pi(j)<\pi(b)$,
and that if $i<a<j$ then we do not have $\pi(b) < \pi(j) <\pi(a)$.
Thus, either $\pi(j)<\pi(b)$ or both $i<a<j$ and $\pi(i)<\pi(b)\leq\pi(a)<\pi(j)$.
The latter conditions contradict Lemma~\ref{lem-2},
so $\pi(j)<\pi(b)$ as desired.
Property (4) holds by a symmetric argument.
\end{proof}

\begin{lemma}\label{tau-thm-lem2}
Assume the conditions in Hypothesis~\ref{lem-hyp}.
In addition, suppose that $i \equiv  y(j) \modu n)$ and $i < y(i) < y(j) < j$. Then $\pi t_{ij} \in \cA(z)$.
\end{lemma}

\begin{proof}
Let $i_0 = i$ and $j_0 = y(i)$ and define $i_m = i_0 + mn$ and $j_m =j_0 + mn$ for $m \in \ZZ$.
Let $A = \{ (i_m, j_m) : m \in \ZZ\}$, $B = \{ (i_m,j_{m+1}) : m \in \ZZ\}$, and $C = \{ (a,b) \in \Cyc(y) : a,b \notin \{i,j\} + n\ZZ\}$.
Then $\Cyc(y) = A \sqcup C$ and $\Cyc(z) = B \sqcup C$.
Theorem~\ref{local-thm} implies that 
\[\pi(j_m) < \pi(i_m) < \pi(j_{m+1}) < \pi(i_{m+1})\qquad\text{for all $m \in \ZZ$.}\]
Let $t=t_{ij}$.
Since $\pi \lessdot \pi t$,
Lemma~\ref{lem-2} implies that $j = j_1$ and $y(j) = i_1$. We have
\[\pi t(i_{m-1}) < \pi t(j_{m+1}) < \pi t(i_m) < \pi t(j_{m+2})\qquad\text{for all $m \in \ZZ$}\]
and $\pi(a) = \pi t(a)$ and $y(a) = z(a)$ for all integers $a \notin \{i,j\} +n \ZZ$.
The conditions in Theorem~\ref{local-thm} therefore hold for $\pi t$ relative to $z$
for all cycles $(a,b),(a',b') \in B$ and all cycles $(a,b),(a',b') \in C$.
Hence, to show that $\pi t \in \cA(z)$, it suffices by Theorem~\ref{local-thm} to check 
for all $(a,b) \in C$ that:
\ben
\item[(1)] If $i <a \leq b < j$ then we do not have $\pi(i) < \pi(a) < \pi(j)$ or $\pi(i) < \pi(b) < \pi(j)$.

\item[(2)] If $a < i < j < b$ then we do not have $\pi(b) < \pi(i) < \pi(a) $ or $\pi(b) < \pi(j) < \pi(a)$.

\item[(3)] If $i < a$ and $j< b$ then $\pi(j) < \pi(b)$.

\item[(4)] If $a < i$ and $b < j$ then $\pi(a) < \pi(i)$.
\een
Property (1) holds by Lemma~\ref{lem-2} since $\pi \lessdot \pi t$.
Property (2) follows from Theorem~\ref{local-thm} since $\pi \in \cA(y)$ and
if $a < i < j < b$ then $a<i<y(i)<b$ and $a<y(j) < j < b$.
To check property (3), suppose $i<a$ and $j<b$. Theorem~\ref{local-thm} implies that 
$\pi(i) < \pi(b)\leq \pi(a)$, that if $y(j) < a$ then $\pi(j) < \pi(y(j)) < \pi(b)$,
and that if $i<a<y(j)$ then we do not have $\pi(b) < \pi(j) <\pi(a)$.
Thus, either $\pi(j) < \pi(b)$ or both $i<a<y(j)<j$ and $\pi(i) < \pi(b) \leq \pi(a) < \pi(j)$.
The latter conditions contradict Lemma~\ref{lem-2},
so $\pi(j) < \pi(b)$ as desired.
Property (4) holds by a symmetric argument.
\end{proof}

\begin{lemma}\label{tau-thm-lem3}
Assume the conditions in Hypothesis~\ref{lem-hyp}.
In addition, suppose that  $i \equiv  y(j) \modu n)$ and $i < y(j) < y(i) < j$. Then $\pi t_{ij} \in \cA(z)$.
\end{lemma}

\begin{proof}
Let $i_0 = i$ and $j_0 = y(i)$ and define $i_m = i_0 + mn$ and $j_m =j_0 + mn$ for $m \in \ZZ$.
Let $A = \{ (i_m, j_m) : m \in \ZZ\}$, $B = \{ (i_m,j_{m+1}) : m \in \ZZ\}$, and $C = \{ (a,b) \in \Cyc(y) : a,b \notin \{i,j\} + n\ZZ\}$.
Then $\Cyc(y) = A \sqcup C$ and $\Cyc(z) = B \sqcup C$.
Theorem~\ref{local-thm} implies that 
\[\pi(j_m) < \pi(i_m) < \pi(i_{m+1}) < \pi(j_{m+1})\qquad\text{for all $m \in \ZZ$.}\]
Let $t=t_{ij}$.
Since $\pi \lessdot \pi t$,
Lemma~\ref{lem-2} implies that $j = j_1$ and $y(j) = i_1$. We have
\[\pi t(j_{m+1}) < \pi t(i_m) < \pi t(j_{m+2}) < \pi t(i_{m+1})\qquad\text{for all $m \in \ZZ$}\]
and $\pi(a) = \pi t(a)$ and $y(a) = z(a)$ for all integers $a \notin \{i,j\} +n \ZZ$.
The conditions in Theorem~\ref{local-thm} therefore hold for $\pi t$ relative to $z$
for all cycles $(a,b),(a',b') \in B$ and all cycles $(a,b),(a',b') \in C$.
Hence, to show that $\pi t \in \cA(z)$, it suffices by Theorem~\ref{local-thm} to check 
for all $(a,b) \in C$ that:
\ben
\item[(1)] If $i <a \leq b < j$ then we do not have $\pi(i) < \pi(a) < \pi(j)$ or $\pi(i) < \pi(b) < \pi(j)$.

\item[(2)] If $a < i < j < b$ then we do not have $\pi(b) < \pi(i) < \pi(a) $ or $\pi(b) < \pi(j) < \pi(a)$.

\item[(3)] If $i < a$ and $j< b$ then $\pi(j) < \pi(b)$.

\item[(4)] If $a < i$ and $b < j$ then $\pi(a) < \pi(i)$.
\een
Property (1) holds by Lemma~\ref{lem-2} since $\pi \lessdot \pi t$.
Property (2) follows from Theorem~\ref{local-thm} since $\pi \in \cA(y)$ and
if $a < i < j < b$ then $a<i<y(i)<b$ and $a<y(j) < j < b$.
To check property (3), suppose $i<a$ and $j<b$. Theorem~\ref{local-thm} implies that 
$\pi(i) < \pi(b)\leq \pi(a)$, that if $y(j) < a$ then $\pi(j) < \pi(y(j)) < \pi(b)$,
and that if $i<a<y(j)$ then we do not have $\pi(b) < \pi(j) <\pi(a)$.
Thus, either $\pi(j) < \pi(b)$ or both $i<a<y(j)<j$ and $\pi(i) < \pi(b) \leq \pi(a) < \pi(j)$.
The latter conditions contradict Lemma~\ref{lem-2},
so $\pi(j) < \pi(b)$ as desired.
Property (4) holds by a symmetric argument.
\end{proof}

\subsection{Computer-assisted arguments}

In principle, the arguments needed to resolve the remaining cases
in the proofs of Theorems~\ref{tau-thm1} and \ref{tau-thm3}
are entirely analogous to the methods in the previous section.
In practice, however, 
these arguments are too complicated to carry out by hand.
We explain in this section how to convert our analysis into a finite computer
calculation.

Recall that $\PP$ is the set of positive integers. Given $m \in \PP$, let $[\pm m] = \{\pm 1,\pm 2,\dots,\pm m\}$.
\begin{definition}
Fix $m \in \PP$ and let $P$, $Q$, and $R$ be formal symbols.
A \emph{virtual permutation} of rank $m$ is a tuple $(\varpi, \cM, \cD, \cS)$ where $\varpi \in S_m$ and
$\cM$, $\cD$, $\cS$ are maps of the following types:
\begin{itemize}
\item $\cM$ is a map from linear extensions of the orders 
$\{1 \prec 2 \prec\dots\prec m\}$, $\{ -1 \prec -2\prec \dots \prec -m\}$, and $\{ -1 \prec 1, -2 \prec 2, \dots, -m \prec m\}$
to sets of linear extensions of the orders 
$\{\varpi_1 \prec \varpi_2 \prec  \dots \prec \varpi_m\}$, $\{-\varpi_1 \prec -\varpi_2 \prec  \dots \prec -\varpi_m\}$,
 and $\{ -1 \prec 1, -2 \prec 2, \dots, -m \prec m\}$.

\item $\cD$ is a map from linear extensions of $\{1\prec 2 \prec\dots\prec m\}$ and $\{ P \prec Q \}$
to sets of linear orders of $[m]\sqcup \{ P,Q\}$ extending $\{\varpi_1 \prec \varpi_2 \prec  \dots \prec \varpi_m\}$.

\item $\cS$ is a map from linear orders of  $[m]\sqcup \{ R \}$ extending $\{1\prec 2\prec\dots\prec m\}$
to sets of linear orders of $[m]\sqcup \{ R \}$ extending $\{\varpi_1 \prec \varpi_2 \prec  \dots \prec \varpi_m\}$.

\end{itemize}
\end{definition}
Let $\vartriangleleft$ denote a generic element of the domain of $\cM$, $\cD$, or $\cS$.
For $\sigma \in S_m$, 
define $\sigma \cM$ to be the map with the same domain as $\cM$, in which $(\sigma \cM)({\vartriangleleft})$
is obtained by applying the transformation with $i \mapsto \sigma(i)$ and $-i \mapsto -\sigma(i)$ for $i \in [m]$
to each linear order in $\cM({\vartriangleleft})$.
For example, if $\sigma = 132$ and $\cM({\vartriangleleft}) = \{ \{ -3 \prec 3 \prec -2 \prec -1 \prec 2 \prec 1\}, \{ -3 \prec -2 \prec -1 \prec 3 \prec 2 \prec 1\}\}$
then
\[
(\sigma\cM)({\vartriangleleft}) = \{ \{ -2 \prec 2 \prec -3 \prec -1 \prec 3 \prec 1\}, \{ -2 \prec -3 \prec -1 \prec 2 \prec 3 \prec 1\}\}
\]
Similarly, define $\sigma\cD$ and $\sigma\cS$ to be the maps with the same domains as $\cD$ and $\cS$,
in which the sets $(\sigma\cD)({\vartriangleleft})$ and $(\sigma\cS)({\vartriangleleft})$ are obtained 
by applying the transformation with $i \mapsto \sigma(i)$ for $i \in [m]$ and $P \mapsto P$ and $Q \mapsto Q$ and $R\mapsto R$
to each linear order in $\cD({\vartriangleleft})$ and $\cS({\vartriangleleft})$, respectively.
Finally, given a virtual permutation $(\varpi,\cM,\cD,\cS)$ of rank $m$ and $\sigma\in S_m$, 
let \[(\varpi,\cM,\cD,\cS)\cdot \sigma = (\sigma^{-1}\varpi, \sigma^{-1}\cM, \sigma^{-1}\cD, \sigma^{-1}\cS).\]
This defines a right action of $S_m$ on the set of virtual permutations of rank $m$.

Fix an involution $y \in I_m$ and define $\Cyc(y) = \{ (a,b) \in [m]\times [m] : a \leq b = y(a)\}$.
A virtual permutation $(\varpi,\cM,\cD,\cS)$ of rank $m$ is a \emph{virtual atom} for $y$
if  
the following conditions hold:
\ben
\item One has $\varpi^{-1} \in \cA(y)$.

\item For each linear order $\vartriangleleft$ in the domain of $\cM$, each linear order ${\prec}\in\cM({\vartriangleleft})$, and
each $(a,b),(a',b') \in \Cyc(y)$ the following conditions hold:
\ben
\item If  $-a \vartriangleleft a'$ and $-b \vartriangleleft b'$ then $-a \prec b'$.
\item If  $a \vartriangleleft -a'$ and $b \vartriangleleft -b'$ then $a \prec -b'$.
\een

\item For each linear order $\vartriangleleft$ in the domain of $\cD$ and each linear order ${\prec}\in\cD({\vartriangleleft})$, 
we have $Q \prec P$ and for each $(a,b) \in \Cyc(y)$ the following conditions hold:
\ben
\item If $a \vartriangleleft P \vartriangleleft Q \vartriangleleft b$ then we do not have $b \prec P \prec a$ or $b\prec Q \prec a$.
\item If $P \vartriangleleft a \vartriangleleft b \vartriangleleft Q$ then we do not have $Q \prec a \prec P$ or $Q\prec b \prec P$.
\item If  $a \vartriangleleft P $ and $b\vartriangleleft Q$ then $a \prec Q$.
\item If  $P \vartriangleleft a$ and  $Q \vartriangleleft b$ then $P \prec b$.
\een

\item For each linear order $\vartriangleleft$ in the domain of $\cD$, each linear order ${\prec}\in\cD({\vartriangleleft})$, 
 and each $(a,b) \in \Cyc(y)$, the following conditions hold:
\ben
\item If $a \vartriangleleft R \vartriangleleft b$ then we do not have $b \prec R \prec a$ or $b\prec R \prec a$.
\item If $b\vartriangleleft R$ then $a \prec R$.
\item If $R \vartriangleleft a$ and  $R \prec b$.
\een
\een

Suppose $\Pi = (\varpi,\cM,\cD,\cS)$ is a virtual permutation of rank $m$. Let $\cE = \cM \sqcup \cD \sqcup \cS$
and define $\Cov(\Pi)$ to be the set of pairs $(i,j)\in [m]\times [m]$ with $i<j$ and $\varpi^{-1}(i) < \varpi^{-1}(j)$,
such that $i$ and $j$ are consecutive in  
each linear order ${\prec} \in \cE({\vartriangleleft})$ for each linear order $\vartriangleleft$ in the domain of $\cE$,
and such that $-i$ and $-j$ are consecutive in each linear order ${\prec} \in \cM({\vartriangleleft})$ for each linear order $\vartriangleleft$ in the domain of $\cM$.

Suppose $\Pi' = (\varpi',\cM',\cD',\cS')$ is another virtual permutation and
$\cE' = \cM' \sqcup \cD' \sqcup \cS'$.
We say that $\Pi'$ contains $\Pi$ and write $\Pi \subset \Pi'$ if $\varpi =\varpi'\in S_m$ and $\cE({\vartriangleleft}) \subset \cE'({\vartriangleleft})$ for each linear order
in the common domain of $\cE$ and $\cE'$.
This definition of containment gives rise to obvious notions of 
virtual permutations that are minimal or maximal with respect to a given property.
If $\sigma \in S_m$ and $\Pi \subset \Pi'$ then clearly $\Pi \cdot \sigma \subset \Pi' \cdot \sigma$.
Likewise, if $\Pi'$ is a virtual atom for $y \in I_m$ and $\Pi \subset \Pi'$,
then $\Pi$ is also a virtual atom for $y$.

We can now reduce our last technical lemma to a finite calculation.

\begin{lemma}\label{tau-thm-lem4}
Assume the conditions in Hypothesis~\ref{lem-hyp}.
In addition, suppose that the sets $\{i,y(i)\} + n \ZZ$ and $\{j,y(j)\} + n\ZZ$  are disjoint. Then $\pi t_{ij} \in \cA(z)$.
\end{lemma}

\begin{proof}
Let $E = \{i,j,y(i),y(j)\}$ and $m=|E|$. 
Write $\phi_E : [m] \to E$ and $\psi_E : E \to [m]$ for the corresponding order-preserving bijections.
The elements of $E$ represent distinct congruence classes modulo $n$ and $y(E) = E$.
Writing $a$, $b$ for generic elements of $[m]$ or $[\pm m]$,
we define $\V(\pi,y,E) = (\varpi,\cM,\cD,\cS)$ to be the minimal virtual permutation with the following properties:
\begin{itemize}
\item $\varpi$ is the permutation in $S_m$ with $\varpi^{-1}(a) < \varpi^{-1}(b)$ if and only if $\pi \circ \phi_E(a) < \pi \circ\phi_E(b)$.

\item For each $k \in \PP$, if $\theta : [\pm m] \to (E-kn) \sqcup E$ 
has $\theta(a) = \phi_E(a)$ and $\theta(-a) = \phi_E(a) - kn$ for $a \in [m]$,
 and $\vartriangleleft$ is the linear order of $ [\pm m]$
with $a \vartriangleleft b$ if and only if $\theta(a) < \theta(b)$,
then $\cM({\vartriangleleft})$ contains the linear order $\prec$ of $[\pm m]$
with $a \prec b$ if and only if $\pi\circ \theta(a) < \pi\circ \theta(b)$.

\item For all integers $p,q \notin E + n \ZZ$ with $p < q=y(q)$, 
if $\theta  : [m]\sqcup \{P,Q\} \to E \sqcup\{p,q\}$ is the bijection with $\theta(a) = \phi_E(a)$ for $a \in [m]$ 
and $\theta(P) = p$ and $\theta(Q) = q$, and $\vartriangleleft$ is the linear order of $[m]\sqcup\{P,Q\}$
with $a\vartriangleleft b$ if and only if $\theta(a) < \theta(b)$,
then $\cD({\vartriangleleft})$ contains the linear order $\prec$ of $[m]\sqcup\{P,Q\}$
with $a \prec b$ if and only if $\pi\circ \theta(a) < \pi\circ \theta(b)$.

\item For all integers $r \notin E + n \ZZ$ with $r=y(r)$, 
if $\theta  : [m]\sqcup \{R\} \to E \sqcup\{r\}$ is the bijection with $\theta(a) = \phi_E(a)$ for $a \in [m]$ 
and $\theta(R) = r$, and $\vartriangleleft$ is the linear order of $[m]\sqcup\{R\}$
with $a\vartriangleleft b$ if and only if $\theta(a) < \theta(b)$,
then $\cS({\vartriangleleft})$ contains the linear order $\prec$ of $[m]\sqcup\{R\}$
with $a \prec b$ if and only if $\pi\circ \theta(a) < \pi\circ \theta(b)$.
\end{itemize}
Let $y' = \psi_E \circ y \circ \phi_E \in I_m$
and  $i' =\psi_E(i)$ and $j' = \psi_E(j)$.
Theorem~\ref{local-thm} and Corollary~\ref{local-cor} imply that $\V(\pi,y,E)$
is a virtual atom for $y'$
and
Lemma~\ref{lem-2} implies
that $(i',j') \in \Cov(\V(\pi,y,E))$.
Since $y(a) = z(a)$ and $w(a) = wt_{ij}(a)$ for all integers $a \notin E + n\ZZ$,
we have $\V(\pi,y,E)\cdot(i',j') = \V(\pi t_{ij}, z, E)$ and
it follows from
Theorem~\ref{local-thm}
that if this virtual permutation is a virtual atom for $z' := \psi_E \circ z \circ \phi_E \in I_m$
then $\pi t_{ij} \in \cA(z)$.
It therefore suffices to show that if $\Pi$ is a virtual atom for $y'$ that 
is maximal among those with $(i',j') \in \Cov(\Pi)$,
then $\Pi\cdot (i',j')$ is a virtual atom for $z'$. 

This is a finite calculation:
there are only 12 possibilities for $y'$, $z'$, and  $1\leq i'< j' \leq m$;
 in each case
there is only one maximal virtual atom $\Pi$ for $y'$ with $(i',j') \in \Cov(\Pi)$;
and it is a straightforward calculation to check that 
 $\Pi\cdot (i',j')$
is a virtual atom for $z'$.
For example, if $y' = (1)(2) \in I_2$ and $i'=1$ and $j'=2$, then $z' = (1,2)$
and we must have $\Pi = (\varpi, \cM,\cD,\cS)$ where
 \ben
 \item[] $\varpi = 12$,
 \item[] $\cM = \begin{cases} \{ -1 \vartriangleleft 1 \vartriangleleft -2 \vartriangleleft 2\} \mapsto \varnothing \\
 \{ -1 \vartriangleleft -2 \vartriangleleft 1 \vartriangleleft 2\} \mapsto \{ \{-1 \prec -2 \prec 1 \prec 2\}\},
 \end{cases}$
 
 \item[] $\cD = \begin{cases} 
 \{  1\vartriangleleft 2 \vartriangleleft P \vartriangleleft Q\} \mapsto \{ \{1 \prec 2 \prec Q \prec P\}\} \\
\{  1\vartriangleleft  P \vartriangleleft 2 \vartriangleleft Q\} \mapsto \{ \{1 \prec 2 \prec Q \prec P\}\} \\
\{ P \vartriangleleft 1\vartriangleleft 2 \vartriangleleft Q\} \mapsto \{ \{1 \prec 2 \prec Q \prec P\},\{1 \prec Q \prec P \prec 2\},\{Q \prec P \prec 1 \prec 2\}\} \\
\{  1 \vartriangleleft P\vartriangleleft Q \vartriangleleft 2\} \mapsto \varnothing \\
\{  P \vartriangleleft  1 \vartriangleleft Q \vartriangleleft 2\} \mapsto \{ \{Q \prec P \prec 1 \prec 2\}\} \\
\{  P \vartriangleleft Q\vartriangleleft 1 \vartriangleleft 2\} \mapsto \{ \{Q \prec P \prec 1 \prec 2\}\},
 \end{cases}$
 
 \item[] $\cS = \begin{cases}
 \{ 1 \vartriangleleft 2 \vartriangleleft R\} \mapsto \{1\prec 2 \prec R\} \\
 \{1 \vartriangleleft R \vartriangleleft 2 \} \mapsto \varnothing \\
 \{ R \vartriangleleft 1 \vartriangleleft 2\} \mapsto \{R \prec 1 \prec 2\}.
 \end{cases}$
 \een
In this case we have $\Pi \cdot (i',j')= (\varpi',\cM',\cD',\cS')$ where
 \ben
 \item[] $\varpi' = 21$,
 \item[] $\cM' = \begin{cases} \{ -1 \vartriangleleft 1 \vartriangleleft -2 \vartriangleleft 2\} \mapsto \varnothing \\
 \{ -1 \vartriangleleft -2 \vartriangleleft 1 \vartriangleleft 2\} \mapsto \{ \{-2 \prec -1 \prec 2 \prec 1\}\},
 \end{cases}$
 
 \item[] $\cD' = \begin{cases} 
 \{  1\vartriangleleft 2 \vartriangleleft P \vartriangleleft Q\} \mapsto \{ \{2 \prec 1 \prec Q \prec P\}\} \\
\{  1\vartriangleleft  P \vartriangleleft 2 \vartriangleleft Q\} \mapsto \{ \{2 \prec 1 \prec Q \prec P\}\} \\
\{ P \vartriangleleft 1\vartriangleleft 2 \vartriangleleft Q\} \mapsto \{ \{2 \prec 1 \prec Q \prec P\},\{2 \prec Q \prec P \prec 1\},\{Q \prec P \prec 2 \prec 1\}\} \\
\{  1 \vartriangleleft P\vartriangleleft Q \vartriangleleft 2\} \mapsto \varnothing \\
\{  P \vartriangleleft  1 \vartriangleleft Q \vartriangleleft 2\} \mapsto \{ \{Q \prec P \prec 2 \prec 1\}\} \\
\{  P \vartriangleleft Q\vartriangleleft 1 \vartriangleleft 2\} \mapsto \{ \{Q \prec P \prec 2 \prec 1\}\},
 \end{cases}$
 
 \item[] $\cS' = \begin{cases}
 \{ 1 \vartriangleleft 2 \vartriangleleft R\} \mapsto \{2\prec 1 \prec R\} \\
 \{1 \vartriangleleft R \vartriangleleft 2 \} \mapsto \varnothing \\
 \{ R \vartriangleleft 1 \vartriangleleft 2\} \mapsto \{R \prec 2 \prec 1\}.
 \end{cases}$
 \een
One can check directly that this is a virtual atom for $z'$.

The relevant analysis for the other 11 cases is similar, but too complicated 
to carry out by hand.
We have written a computer program  to 
enumerate the possible cases and check the required conditions;
our code also generates a human readable, but extremely long and tedious proof of this lemma,
which is available as a 162 page {\tt pdf} file \cite{code}.
Our program's computations show that in every case $\Pi \cdot (i',j')$ is a virtual atom for $z'$,
so we conclude that $wt_{ij} \in \cA(z)$ as desired.
\end{proof}

Combining the preceding lemmas, we may finally prove
Theorem~\ref{tau-thm1}.

\begin{proof}[Proof of Theorem~\ref{tau-thm1}]
Let
$y,z \in \tI_n$ and $\pi \in \cA(y)$. Fix integers  $i<j \not \equiv i \modu n)$
 such that $\pi \lessdot \pi t_{ij}$.
The assertion that $z = \tau^n_{ij}(y) \neq y$ if $\pi t_{ij} \in \cA(z)$ is 
 \cite[Theorem 8.10]{M}.
Conversely, suppose $z = \tau^n_{ij}(y)\neq y$.
If $i \equiv y(j) \modu n)$, then $y$ must be as in Lemmas~\ref{tau-thm-lem1}, \ref{tau-thm-lem2},
or \ref{tau-thm-lem3}, and in these cases we have $\pi t_{ij} \in\cA(z)$ as desired.
If $i \not\equiv y(j)\modu n)$, then the sets
$\{i,y(i)\} + n \ZZ$ and $\{j,y(j)\} + n\ZZ$  are disjoint, so we have $\pi t_{ij} \in \cA(z)$ by Lemma~\ref{tau-thm-lem4}.
\end{proof}

Using similar methods, we can also prove the toggling property described in Theorem~\ref{tau-thm3}.
Suppose $y \in \tI_n$ and $\pi \in \cA(y)$.
Fix integers  $i<j \not \equiv i \modu n)$ such that $\pi \lessdot \pi t_{ij}$ and $y=\tau^n_{ij}(y)$.
Define $k \in \{j, y(j)\}$ and $l \in \{i,y(i)\}$ as in Theorem~\ref{tau-thm3}.

\begin{proof}[Proof of Theorem~\ref{tau-thm3}]
It holds by inspection that $k<l\not\equiv k \modu n)$ and $t_{ij} \neq t_{kl}$,
and it follows from Theorem~\ref{tau-thm1} that $y=\tau^n_{kl}(y)$.
Let $E = \{i,j,y(i),y(j)\} = \{k,l,y(k),y(l)\}$ and $m=|E|$. 
Write $\phi_E : [m] \to E$ and $\psi_E : E \to [m]$ for the corresponding order-preserving bijections.
The elements of $E$ represent distinct congruence classes modulo $n$ and $y(E) = E$.
Define \[\V(\pi,y,E) = (\varpi,\cM,\cD,\cS)\] as in the proof of Theorem~\ref{tau-thm-lem4}.
Let $y' = \psi_E \circ y \circ \phi_E \in I_m$, $i' =\psi_E(i)$, $j' = \psi_E(j)$, $k'=\psi_E(k)$, and $l'=\psi_E(l)$.
Theorem~\ref{local-thm} and Corollary~\ref{local-cor} imply that $\V(\pi,y,E)$
is a virtual atom for $y'$
and
Lemma~\ref{lem-2} implies
that $(i',j') \in \Cov(\V(\pi,y,E))$.
We have \[\V(\pi t_{ij}t_{kl},y,E) = \V(\pi,y,E)\cdot (i',j')(k',l')\] by construction,
and it follows from
Theorem~\ref{local-thm}
that if this virtual permutation is also a virtual atom for $y' $
then $\pi t_{ij}t_{kl} \in \cA(y)$.
It therefore suffices to show that if $\Pi$ is virtual atom for $y'$ 
that is maximal among those with $(i',j') \in \Cov(\Pi)$,
then $\Pi\cdot (i',j')(k',l')$ is a virtual atom for $y$. 

This again reduces to a finite calculation.
There are 8 cases for the involution $y'$, the indices $1\leq i' < j' \leq m$,  and 
the virtual atom $\Pi$,
corresponding to the parts of Proposition~\ref{tech-prop}. 
As in the proof of Theorem~\ref{tau-thm-lem4},
we have written a computer program  
that can enumerate these cases 
and check the required conditions, and which also produces a very long but human readable proof of the theorem (accessible as a 134 page \textsf{pdf} document) \cite{code}.
Our program's computations show that in every case 
$\Pi\cdot (i',j')(k',l')$ is a virtual atom for $y'$,
so we conclude that $\pi t_{ij}t_{kl} \in \cA(y)$.
\end{proof}

\end{document}